\documentclass[oneside,english]{amsart}
\usepackage[T1]{fontenc}
\usepackage[latin9]{inputenc}
\usepackage{color}
\usepackage{babel}
\usepackage{refstyle}
\usepackage{amstext}
\usepackage{amsthm}
\usepackage{amssymb}
\usepackage{stmaryrd}
\usepackage{graphicx}
\usepackage[unicode=true,pdfusetitle,
 bookmarks=true,bookmarksnumbered=false,bookmarksopen=false,
 breaklinks=false,pdfborder={0 0 1},backref=false,colorlinks=true]
 {hyperref}
\hypersetup{
 linkcolor=red, urlcolor=red, citecolor=red}

\makeatletter

\AtBeginDocument{\providecommand\thmref[1]{\ref{thm:#1}}}
\AtBeginDocument{\providecommand\propref[1]{\ref{prop:#1}}}
\AtBeginDocument{\providecommand\eqref[1]{\ref{eq:#1}}}
\AtBeginDocument{\providecommand\lemref[1]{\ref{lem:#1}}}
\RS@ifundefined{subsecref}
  {\newref{subsec}{name = \RSsectxt}}
  {}
\RS@ifundefined{thmref}
  {\def\RSthmtxt{theorem~}\newref{thm}{name = \RSthmtxt}}
  {}
\RS@ifundefined{lemref}
  {\def\RSlemtxt{lemma~}\newref{lem}{name = \RSlemtxt}}
  {}

\numberwithin{equation}{section}
\numberwithin{figure}{section}
\theoremstyle{plain}
\newtheorem{thm}{\protect\theoremname}
  \theoremstyle{definition}
  \newtheorem{defn}{\protect\definitionname}
  \theoremstyle{plain}
  \newtheorem{prop}{\protect\propositionname}
  \theoremstyle{plain}
  \newtheorem{lem}{\protect\lemmaname}

\makeatother

  \providecommand{\definitionname}{Definition}
  \providecommand{\lemmaname}{Lemma}
  \providecommand{\propositionname}{Proposition}
\providecommand{\theoremname}{Theorem}

\begin{document}

\title[The wave equation outside two convex obstacles]{About the wave equation outside two strictly convex obstacles }

\author{David Lafontaine}
\begin{abstract}
We prove global Strichartz estimates without loss for the wave equation
outside two strictly convex obstacles, following the roadmap introduced
in \cite{Schreodinger} for the Schrödinger equation. Moreover, we
show a first step toward the large data scattering for the critical
non linear equation associated to this geometrical setting, and prove
the scattering for a class of non trapping obstacles close to the
two convex framework.
\end{abstract}

\maketitle

\section{Introduction}

Let $(M,g)$ be a Riemaniann manifold of dimension $d$. We are interested
in the linear wave equation on $M$
\begin{align}
\begin{cases}
\partial_{t}^{2}u-\Delta_{g}u=0\\
(u(0),\partial_{t}u(0))=(f,g).
\end{cases}\label{eq:lw}
\end{align}
where $\Delta_{g}$ design the Laplace-Beltrami operator. In order
to study the perturbative theory and the nonlinear problems associated
with this equation, it is crucial to estimate the size and the decay
of the solutions. Such estimates are the so called \textit{Strichartz
estimates}
\begin{equation}
\Vert u\Vert_{L^{p}(0,T)L^{q}(\Omega)}\leq C_{T}\left(\Vert u_{0}\Vert_{\dot{H}^{\gamma}}+\Vert u_{1}\Vert_{\dot{H}^{\gamma-1}}\right),\label{eq:Strichoc1}
\end{equation}
where $(p,q)$ has to follow the admissibility condition given by
the scaling of the equation
\begin{equation}
\frac{1}{p}+\frac{d}{q}=\frac{d}{2}-\gamma,\label{eq:scaling}
\end{equation}
and
\begin{equation}
\frac{1}{p}\leq\frac{d-1}{2}\left(\frac{1}{2}-\frac{1}{p}\right).\label{eq:inwl}
\end{equation}
We say that the estimates hold with a loss of order $\lambda>0$ if
they hold for $(p,q)$ satisfying the scaling condition (\ref{eq:scaling}),
and
\[
\frac{1}{p}\leq\left(\frac{d-1}{2}-\lambda\right)\left(\frac{1}{2}-\frac{1}{p}\right).
\]

Strichartz estimates were first introduced and established in \cite{Strichartz}
for the $p=q$ case in $\mathbb{R}^{n}$, then extended to all exponents
in \cite{GV85}, \cite{LindbladSogge}, and \cite{KeelTao}. As usual,
the variable coefficient case is more difficult. In the case of a
manifold without boundary, the finite speed of propagation shows that
it suffises to obtain the estimates in local coordinates to obtain
local Strichartz estimates. Such estimates were obtained by \cite{Kapitanskii},
\cite{MoSeSo}, \cite{SmithC11}, and \cite{Tataruns}. The estimates
outside one convex obstacle were obtained by \cite{SS95}, following
the parametrix construction of Melrose and Taylor. Local estimates
on a general domain were first proved by \cite{BLP} for certain ranges
of $(p,q)$, using spectral estimates of \cite{SSspectral}. The range
of indices was then extended by \cite{BSS}. This range cannot recover
all indices satisfying (\ref{eq:inwl}) : \cite{OanaCounterex} showed
indeed that a loss have to occur if some concavity is met. Recently,
\cite{ILPAnnals} proved in a model case local Strichartz estimates
inside a convex domain with a loss close to the sharpest one. Their
result is extended in \cite{ILLPGeneral} to the wave equation.

Phenomenons such as closed geodesics can be obstacles to the establishment
of global estimates. Under a non trapping asumption, \cite{SmithSoggeNonTrapping}
proved in the odd dimensional case that local estimates can be extended
to global ones. This result was extended to the even dimensions independently
by \cite{MR2001179} and \cite{Metcalfe}.

However, \cite{MR2720226} showed that Strichartz estimates without
loss for the Schrödinger equation hold for an asymptitocaly euclidian
manifold without boundary for which the trapped set is sufficently
small and exhibit an hyperbolic dynamic. 

Going in the same direction for the problem with boundaries, we recently
showed in \cite{Schreodinger} global Strichartz estimates without
loss for the Schrödinger equation outside two convex obstacles. The
aim of this paper is to extend this result to the wave equation. More
precisely, we prove
\begin{thm}
\label{thm:main}Let $\Theta_{1}$ and $\Theta_{2}$ be two compact,
strictly convex subsets of $\mathbb{R}^{n}$, $u$ be a solution of
(\ref{eq:lw}) in $\Omega=\mathbb{R}^{n}\backslash\left(\Theta_{1}\cup\Theta_{2}\right)$
and $(p,q,\gamma)$ verifying (\ref{eq:scaling}) and (\ref{eq:inwl}).
Then
\begin{equation}
\Vert u\Vert_{L^{p}(\mathbb{R},L^{q})}\leq C\left(\Vert f\Vert_{\dot{H}^{\gamma}}+\Vert g\Vert_{\dot{H}^{\gamma-1}}\right).\label{eq:striuch}
\end{equation}
\end{thm}
The crucial remark of \cite{SmithSoggeNonTrapping} is that local
Strichartz estimates combined with the exponential decay of the energy
permits to obtain global Strichartz estimates. For the exterior of
one convex obstacle in odd dimension, this decay holds and global
Strichartz estimates without loss are obtained. In even dimension,
such an exponential decay do not hold anymore. However, \cite{MR2001179}
remarked that it can be remplaced by weaker estimates of $L^{2}-$integrability
of the local energy
\begin{equation}
\Vert(\chi u,\chi\partial_{t}u)\Vert_{L^{2}(\mathbb{R},L^{2}\times H^{-1})}\lesssim\Vert u_{0}\Vert_{\dot{H}^{\gamma}}+\Vert u_{1}\Vert_{\dot{H}^{\gamma-1}},\label{eq:L2int}
\end{equation}
where $\chi$ is any compactly supported function, and such an estimate
for the complementary of a convex obstacle is a direct consequence
of well-known resolvent estimates.

But in the case of the exterior of two convex obstacles, (\ref{eq:L2int})
do not hold anymore: a logarithmic loss occurs due to the trapped
geodesic and we only have
\begin{equation}
\Vert(\chi u,\chi\partial_{t}u)\Vert_{L^{2}(\mathbb{R},L^{2}\times H^{-1})}\lesssim|\log h|\left(\Vert u_{0}\Vert_{\dot{H}^{\gamma}}+\Vert u_{1}\Vert_{\dot{H}^{\gamma-1}}\right).\label{eq:L2int-1}
\end{equation}
for data supported in frequencies $\sim h^{-1}$. The $L^{2}$-integrability
of the local energy is the waves-analog of the \textit{smoothing effect}
for the Schrödinger equation, for which a loss occurs in the same
way. \cite{MR2720226} remarked that such a loss can be compensated
if we show Strichartz estimates in logarithmic times and we followed
this idea in \cite{Schreodinger}. We follow here the same roadmap
and show that this logarithmic loss can be compensated if we show
Strichartz estimates in logarithmic times in the neighborhood of the
trapped ray
\[
\Vert\chi u\Vert_{L^{p}(0,|\log h|)L^{q}}\lesssim\Vert u_{0}\Vert_{L^{2}}+\Vert u_{1}\Vert_{H^{-1}}.
\]
Then, we reduce again the problem, to data which micro-locally contains
only points of the tangent space which do not escape a given neighboorhood
of the periodic ray after logarithmic times. Finally, we construct
an approximate solution for such data, inspired by \cite{Ikawa2,IkawaMult},
\cite{plaques}, and we show that this approximation gives the desired
estimate. 

Note that a large part of the contruction we are doing here is similar
to which we did in \cite{Schreodinger}, and we will extensively use
results of this previous paper. On the one hand, the wave equation
enjoys an exact speed of propagation, and all the results who reliated
on the semiclassical finite speed of propagation of the Schrödinger
flow hold with simplificated proofs. On the other hand, the phases
of the approximate solution we are building stationate now in whole
lines, instead of points, and it is a little more subtle to close
the final argument.

\subsection*{Application}

As an application, we consider a critical defocusing non linear wave
equation in $\mathbb{R}^{3}\backslash\left(\Theta_{1}\cup\Theta_{2}\right)$
\begin{align}
\begin{cases}
\partial_{t}^{2}u-\Delta_{D}u+u^{5}=0\\
(u(0),\partial_{t}u(0))=(f,g).
\end{cases}\label{eq:nlw}
\end{align}
Note that the global existence for such an equation in a domain was
obtained in \cite{BLP}. By the finite speed of propagation, their
result apply in particular to the exterior of obstacles. Therefore,
it is legitimate to wonder what solutions look like in large time,
and in particular if the nonlinearity still plays a role. If it is
not the case, we say that the solution \emph{scatters}. More precisely,
we say that a solution scatters if there exists a solution of the
\emph{linear} equation $v$ such that
\[
\Vert u(t)-v(t)\Vert_{\dot{H}^{1}(\Omega)}\longrightarrow0,
\]
as $t$ goes to infinity. The scattering in $\mathbb{R}^{3}$ was
shown by Bahouri and Shatah \cite{BahSha}. Provided a good set of
Strichartz estimates exists for the linear equation, their proof adapts
to the case of a finite-border domain if one is able to deal with
the arising boundary term. This term can be controled in particular
if one obtain the decay of the local energy near the obstacle (see
Section 5):
\begin{equation}
\frac{1}{T}\int_{0}^{T}\int_{\Omega\cap B(0,A)}|\nabla u(x,t)|^{2}+|u(x,t)|^{6}\ dxdt\longrightarrow0\label{eq:contrenergynobst}
\end{equation}
as $T$ goes to infinity. In the case of the exterior of two balls,
where \thmref{main} gives us the good set of Strichartz estimates,
we show that we can obtain this control \emph{everywhere except in
the neighborhood of the trapped ray}: more precisely
\begin{thm}
\label{th2}Let $\Theta_{1}$ and $\Theta_{2}$ be two disjoint balls
of $\mathbb{R}^{3}$. Then, there exists a family $(\mathcal{S}(T))_{T\geq1}$
of open neighborhoods of the trapped ray $\mathcal{R}$ verifying
\[
\mathcal{S}(T)\longrightarrow\mathcal{R}\ \text{as }T\longrightarrow+\infty
\]
such that any global solution of (\ref{eq:nlw}) in $\Omega:=\mathbb{R}^{3}\backslash\left(\Theta_{1}\cup\Theta_{2}\right)$
verifies, as $T$ goes to infinity
\[
\frac{1}{T}\int_{0}^{T}\int_{\left(\Omega\cap B(0,A)\right)\backslash\mathcal{S}(T)}|\nabla u(x,t)|^{2}+|u(x,t)|^{6}\ dxdt\longrightarrow0.
\]
\end{thm}
This is a first step to show the scattering for any data in this exterior
problem. We are precisely following this path in the work in progress
\cite{LafLaurent}, where this step is extended to the exterior of
two arbitrary convex obstacles and used to show the scattering in
this framework.

We now deal with a geometrical situation which is close to the exterior
of two convex obstacles, but does not have a trapped ray: the exterior
of dog bones. We are actually able to show the scattering outside
a class of non star-shaped obstacles containing dog bones with arbitrary
thin necks. In order to state this result, let us recall the definition
of an illuminated subset - which is a generalization of star-shaped
ones - first introduced by \cite{MR0347237}:
\begin{defn}
\label{defill}A subset $\mathcal{K}$ of $\mathbb{R}^{n}$ is sayed
to be illuminated by a convex subset $\mathcal{C}\subset\mathbb{R}^{n}$
if
\[
\min_{\partial\mathcal{K}}\nabla\rho\cdot\nu>0
\]
where $\rho$ is the gauge of $\mathcal{C}$ and $\nu$ the outward-pointing
normal derivative to $\partial\mathcal{K}$. 
\end{defn}
We are now able to state our result:
\begin{thm}
\label{th3}Let $\mathcal{C}\subset\mathbb{R}^{3}$ be the ellipsoïd
of equation 
\begin{equation}
x^{2}+y^{2}+\epsilon z^{2}=1,\ 0<\epsilon\leq1\label{eq:el1}
\end{equation}
resp.
\begin{equation}
x^{2}+\epsilon y^{2}+\epsilon z^{2}=1,\ \frac{1+\sqrt{3}}{4}\leq\epsilon\leq1\label{eq:el2}
\end{equation}
and $\mathcal{K}$ be a compact subset of $\mathbb{R}^{3}$ illuminated
by $\mathcal{C}$. Then, any solution of (\ref{eq:nlw}) in $\Omega=\mathbb{R}^{3}\backslash\mathcal{K}$
scatters in $\dot{H}^{1}(\Omega)$.
\end{thm}
Notice that Abou-Shakra obtained in \cite{farah} the scattering for
obstacles illuminated by a deformation of a sphere using a slightly
different method, but her result does not permit to handle dog bones
with arbitrary thin necks. Our key tool to obtain theorem \ref{th2}
and theorem \ref{th3} is an identity due to Morawetz \cite{Morawetz61}
in the case of the linear equation, and used here in the spirit of
\cite{GinibreVeloKG85}. Such an identity rely on the choice of a
good weight function $\chi$ which has to be adapted to the geometry
and verify a very rigid and poorly understood constraint: $\Delta^{2}\chi\leq0$.
In the case of theorem \ref{th3}, the natural weight is given by
the gauge of the ellipsoïd we are dealing with, and does not verify
this constraint for arbitrary thin ellipsoïds. In order to deal with
it, we present a method which permits to bypass this obstruction:
noticing that corresponding four dimensional ellipsoïds verify the
constraint, we extend the solution as the solution of a four dimensional
non linear wave equation, show the estimate for such a solution, and
then go back to the original, three dimensional solution. We believe
that such an argument may be useful in other situations.

\section{Reduction of the problem}

\subsection{Estimates of $L^{2}$-integrability of the local energy}

We first show the following two estimates of the $L^{2}-$integrability
of the local energy, that we will need in the sequel. Their are the
analogs of the smoothing estimates for the Schrödinger flow, and were
introduced by \cite{MR2001179} in the non-trapping case. The first
one is an estimate without loss away of the trapped ray. The second
one holds in the whole exterior domain, but with a logarithmic loss.
\begin{prop}[Global $L^{2}$-integrability with no loss away of the trapped ray]
 Let $\chi\in C_{0}^{\infty}$ be supported outside a small enough
neighborhood of the trapped ray. Then, if $u$ is the solution of
(\ref{eq:lw}) with data $(f,g)$:
\begin{equation}
\Vert(\chi u,\chi\partial_{t}u)\Vert_{L^{2}(\mathbb{R},\dot{H}^{\gamma}\times\dot{H}^{\gamma-1})}\lesssim\Vert f\Vert_{\dot{H}^{\gamma}}+\Vert f\Vert_{\dot{H}^{\gamma-1}}\label{eq:globalaway}
\end{equation}
\end{prop}
\begin{proof}
As \cite{MR2001179} show it in section 2, it suffises to obtain such
an estimate to show the resolvent estimate
\[
\Vert\chi(-\Delta_{D}-(\lambda\pm i\epsilon))^{-1}\chi\Vert_{L^{2}\rightarrow L^{2}}\lesssim\frac{1}{1+\sqrt{|\lambda|}}.
\]
In the spirit of \cite{Schreodinger}, let $K$ be a non-trapping
obstacle such that $K$ and $\Theta_{1}\cup\Theta_{2}$ coincide in
the support of $\chi$. In particular, $\Delta_{\Omega}=\Delta_{\mathbb{R}^{n}\backslash K}$
on the support of $\chi$. As, moreover, the resolvent estimate is
well-known in the non-trapping case (see \cite{MR1764368} and \cite{MR0492794,MR644020}
for the high frequencies part, \cite{MR1618254} for the low frequencies),
we have
\[
\Vert\chi(-\Delta_{\Omega}-(\lambda\pm i\epsilon))^{-1}\chi\Vert_{L^{2}\rightarrow L^{2}}=\Vert\chi(-\Delta_{\mathbb{R}^{n}\backslash K}-(\lambda\pm i\epsilon))^{-1}\chi\Vert_{L^{2}\rightarrow L^{2}}\lesssim\frac{1}{1+\sqrt{|\lambda|}},
\]
and the Proposition is shown.
\end{proof}
\begin{prop}[Global $L^{2}$-integrability with logarithmic loss]
\label{prop:smooth_away}Let $\chi\in C_{0}^{\infty}$. Then, if
$f,g$ verifies $\psi(-h^{2}\Delta)f=f$, $\psi(-h^{2}\Delta)g=g$
and $u$ is the solution of (\ref{eq:lw}) with data $(f,g)$:
\begin{equation}
\Vert(\chi u,\chi\partial_{t}u)\Vert_{L^{2}(\mathbb{R},\dot{H}^{\gamma}\times\dot{H}^{\gamma-1})}\lesssim|\log h|^{1/2}\left(\Vert f\Vert_{\dot{H}^{\gamma}}+\Vert g\Vert_{\dot{H}^{\gamma-1}}\right)\label{eq:globalnear}
\end{equation}
\end{prop}
\begin{proof}
Denote
\begin{align*}
\dot{H}^{\gamma,-} & =D((-\Delta_{D})^{s/2}\log(2I-\Delta)^{-1/2}),\\
H^{\gamma,-} & =D((I-\Delta_{D})^{s/2}\log(2I-\Delta)^{-1/2}),
\end{align*}
by $\dot{H}^{-\gamma,+}$ and $H^{-\gamma,+}$ their dual, and
\[
\mathcal{H}^{\gamma,-}=\dot{H}^{\gamma,-}\times\dot{H}^{\gamma-1,-},\ \mathcal{H}^{-\gamma,+}=\dot{H}^{-\gamma,+}\times\dot{H}^{-(\gamma-1),+}.
\]
Finally, let us denote
\[
A=i\begin{pmatrix}0 & -I\\
-\Delta & 0
\end{pmatrix}.
\]
 We will show the estimate
\begin{equation}
\Vert(\chi u,\chi\partial_{t}u)\Vert_{L^{2}(\mathbb{R},\dot{H}^{\gamma,-}\times\dot{H}^{\gamma-1,-})}\lesssim\Vert f\Vert_{\dot{H}^{\gamma}}+\Vert g\Vert_{\dot{H}^{\gamma-1}}.\label{eq:smooth_bourb}
\end{equation}
By a classical $TT^{\star}$ argument (see for example \cite{MR2001179}
and \cite{MR2066943}), (\ref{eq:smooth_bourb}) is a direct consequence
of the following proposition
\begin{prop}
The resolvent $\chi(A-(x-i\epsilon)^{-1})\chi$ is uniformly bounded
in 
\[
\mathcal{H}^{-s,+}\rightarrow\mathcal{H}^{s,-}
\]
 for $x\in\mathbb{R}$ and $0<\epsilon<1$.
\end{prop}
Which, in turn, we will obtain as a consequence of the following resolvent
estimate obtained by \cite{MR2066943}:

\begin{equation}
\Vert\chi(-\Delta_{D}-(\lambda\pm i\epsilon))^{-1}\chi\Vert_{L^{2}\rightarrow L^{2}}\lesssim\frac{2\log(2+|\lambda|)}{1+\sqrt{|\lambda|}}.\label{eq:resol_smooth}
\end{equation}
As 
\[
(A-z)^{-1}=\begin{pmatrix}-z(\Delta+z^{2})^{-1} & i(\Delta+z^{2})^{-1}\\
i\Delta(\Delta+z^{2})^{-1} & -z(\Delta+z^{2})^{-1}
\end{pmatrix},
\]
we have to show that the following norms are uniformly bounded, for
all $s\in\mathbb{R}$
\begin{gather*}
\Vert\chi(1+|z|)(\Delta+z^{2})^{-1}\chi\Vert_{H^{-s,+}\rightarrow H^{s,-}},\\
\Vert\chi(\Delta+z^{2})^{-1}\chi\Vert_{H^{-s,+}\rightarrow H^{s+1,-}},\\
\Vert\chi\Delta(\Delta+z^{2})^{-1}\chi\Vert_{H^{-s,+}\rightarrow H^{s-1,-}}.
\end{gather*}
With the same arguments as \cite{MR2001179}, theses bounds are all
consequences of the first one for $s=0$, that is of
\begin{equation}
\Vert\chi(1+|z|)(\Delta+z^{2})^{-1}\chi\Vert_{\mathcal{H}^{0,+}\rightarrow\mathcal{H}^{0,-}}.\label{eq:redsmooth}
\end{equation}
To show (\ref{eq:redsmooth}), we follow \cite{MR2066943}, Section
4. Let 
\[
u=(1+|z|)(\Delta+z^{2})^{-1}\chi f.
\]
For $\Psi\in C_{0}^{\infty}(-1/2,2)$ equal to one close to $1$,
we decompose
\[
u=\Psi(-\frac{\Delta}{z^{2}})u+\left(1-\Psi(-\frac{\Delta}{z^{2}})\right)u.
\]
On the one hand, 
\[
\text{\ensuremath{\Vert}}\left(1-\Psi(-\frac{\Delta}{z^{2}})\right)u\Vert_{L^{2}}\lesssim\text{\ensuremath{\Vert}}\left(1-\Psi(-\frac{\Delta}{z^{2}})\right)\chi f\Vert_{L^{2}}.
\]
On the other hand, as
\[
\Psi(-\frac{\Delta}{z^{2}})u=(1+|z|)(\Delta+z^{2})^{-1}\Psi(-\frac{\Delta}{z^{2}})\chi f,
\]
we have from (\ref{eq:resol_smooth})
\[
\Vert\chi\Psi(-\frac{\Delta}{z^{2}})u\Vert_{L^{2}}\lesssim\log(2+z^{2})\Vert\Psi(-\frac{\Delta}{z^{2}})\chi f\Vert_{L^{2}},
\]
and thus
\[
\log(2+z^{2})^{-1/2}\Vert\chi\Psi(-\frac{\Delta}{z^{2}})u\Vert_{L^{2}}\lesssim\log(2+z^{2})^{1/2}\Vert\Psi(-\frac{\Delta}{z^{2}})\chi f\Vert_{L^{2}}.
\]
Finally, like in \cite{MR2066943}, the localization in frequencies
allows us to replace the weights in $z$ by the $H^{0,\pm}$ norms,
and we get (\ref{eq:redsmooth}).
\end{proof}

\subsection{Reduction to logarithmic times near the trapped ray}

The aim of this section is to show that the following proposition
implies \thmref{main}
\begin{prop}
\label{prop:red1}There exists $\epsilon>0$ and a small neighborhood
$D$ of the trapped ray, such that, for all $\chi\in C_{0}^{\infty}$
supported in $D$, if $f$, $g$ are such that $\psi(-h^{2}\Delta)f=f$,
$\psi(-h^{2}\Delta)g=g$ and $u$ is the solution of (\ref{eq:lw})
with data $(f,g)$:
\[
\Vert\chi u\Vert_{L^{p}(0,\epsilon|\log h|)L^{q}}\lesssim\Vert f\Vert_{\dot{H}^{\gamma}}+\Vert g\Vert_{\dot{H}^{\gamma-1}}.
\]
\end{prop}
Thus, we will assume the previous Proposition and show \thmref{main}.
As the value of $\epsilon>0$ does not play any role, we assume here
that $\epsilon=1$.

In the spirit of \cite{Schreodinger}, let $\chi_{\text{obst}},\chi_{\text{ray}}\in C_{0}^{\infty}$
be such that $\chi_{\text{obst}}=1$ in a neighborhood of $\Theta_{1}\cup\Theta_{2}\cup\mathcal{R}$,
and $\chi_{\text{ray}}\in C_{0}^{\infty}$ such that $\chi_{\text{ray}}=1$
in a neighborhood of $\mathcal{R}$. We decompose $u$ as the sum
\begin{equation}
u=(1-\chi_{\text{obst}})u+\chi_{\text{obst}}(1-\chi_{\text{ray}})u+\chi_{\text{obst}}\chi_{\text{ray}}u\label{eq:dec-red}
\end{equation}

\subsubsection{The first term: away from the trapped ray and the obstacles}

Let $v=(1-\chi_{\text{obst}})u$. Then $v$ verifies
\begin{align*}
\left(\partial_{t}^{2}-\Delta_{D}\right)v & =-[\Delta_{D},\chi_{\text{obst}}]u,\\
(v(0),\partial_{t}v(0)) & =((1-\chi_{\text{obst}})f,(1-\chi_{\text{obst}})g).
\end{align*}
As $v$ is supported away from the obstacle, it solves a problem in
the full space and we can replace the Laplacian in $\left(\partial_{t}^{2}-\Delta_{D}\right)$
by $\Delta_{\mathbb{R}^{n}}$. Therefore, by the Duhamel formula
\begin{multline}
v(t)=\cos(t\sqrt{-\Delta_{\mathbb{R}^{n}}})(1-\chi_{\text{obst}})f+\frac{\sin t\sqrt{-\Delta_{\mathbb{R}^{n}}}}{\sqrt{-\Delta_{\mathbb{R}^{n}}}}(1-\chi_{\text{obst}})g\\
-\int_{0}^{t}\frac{\sin((t-s)\sqrt{-\Delta_{\mathbb{R}^{n}}})}{\sqrt{-\Delta_{\mathbb{R}^{n}}}}[\Delta_{D},\chi]u(s)ds.\label{eq:2v1}
\end{multline}
The first two terms are handled thanks to the Strichartz estimates
for the waves in $\mathbb{R}^{n}$:
\begin{multline}
\Vert\cos(t\sqrt{-\Delta_{\mathbb{R}^{n}}})(1-\chi_{\text{obst}})f+\frac{\sin t\sqrt{-\Delta_{\mathbb{R}^{n}}}}{\sqrt{-\Delta_{\mathbb{R}^{n}}}}(1-\chi_{\text{obst}})g\Vert_{L^{p}(\mathbb{R},L^{q})}\\
\lesssim\Vert f\Vert_{\dot{H}^{\gamma}}+\Vert g\Vert_{\dot{H}^{\gamma-1}}.\label{eq:2v=0000E9}
\end{multline}
And by Christ-Kiselev lemma, cuting the sinus in half wave operators
and the Strichartz estimates in the full space again
\begin{multline}
\Vert\int_{0}^{t}\frac{\sin((t-s)\sqrt{-\Delta_{\mathbb{R}^{n}}})}{\sqrt{-\Delta_{\mathbb{R}^{n}}}}[\Delta_{D},\chi_{\text{obst}}]u(s)ds\Vert_{L^{p}L^{q}}\\
\lesssim\Vert\int_{\mathbb{R}}\frac{\sin((t-s)\sqrt{-\Delta_{\mathbb{R}^{n}}})}{\sqrt{-\Delta_{\mathbb{R}^{n}}}}[\Delta_{D},\chi_{\text{obst}}]u(s)ds\Vert_{L^{p}L^{q}}\\
\lesssim\Vert\frac{e^{-it\sqrt{-\Delta_{\mathbb{R}^{n}}}}}{\sqrt{-\Delta_{\mathbb{R}^{n}}}}\int_{\mathbb{R}}e^{is\sqrt{-\Delta_{\mathbb{R}^{n}}}}[\Delta_{D},\chi_{\text{obst}}]u(s)ds\Vert_{L^{p}L^{q}}\\
\lesssim\Vert\int_{\mathbb{R}}e^{is\sqrt{-\Delta_{\mathbb{R}^{n}}}}[\Delta_{D},\chi_{\text{obst}}]u(s)ds\Vert_{\dot{H}^{\gamma-1}}.\label{eq:2v3}
\end{multline}
Now, thanks to the dual version of the $L^{2}$ estimate (\ref{eq:globalaway})
in $\mathbb{R}^{n}$ (Proposition \propref{smooth_away} replacing
$\Omega$ by $\mathbb{R}^{n}$, which enjoys no trapped geodesic)
we get
\begin{multline}
\Vert\int_{\mathbb{R}}e^{is\sqrt{-\Delta_{\mathbb{R}^{n}}}}[\Delta_{D},\chi_{\text{obst}}]u(s)ds\Vert_{\dot{H}^{\gamma-1}}\\
=\Vert\int_{\mathbb{R}}e^{is\sqrt{-\Delta_{\mathbb{R}^{n}}}}\tilde{\chi}[\Delta_{D},\chi_{\text{obst}}]u(s)ds\Vert_{\dot{H}^{\gamma-1}}\lesssim\Vert[\Delta_{D},\chi_{\text{obst}}]u(s)ds\Vert_{L^{2}\dot{H}^{\gamma-1}},\label{eq:2v4}
\end{multline}
where $\tilde{\chi}=1$ on the support of $\nabla\chi_{\text{obst}}$
. But, using the $L^{2}$ estimate away from the trapped ray (\ref{eq:globalaway}),
because $\nabla\chi_{\text{obst}}$ is supported away from the trapped
ray:
\begin{multline}
\Vert[\Delta_{D},\chi_{\text{obst}}]u(s)ds\Vert_{L^{2}\dot{H}^{\gamma-1}}\lesssim\Vert f\Vert_{\dot{H}^{\gamma-1}}+\Vert g\Vert_{\dot{H}^{\gamma-2}}+\Vert\nabla f\Vert_{\dot{H}^{\gamma-1}}+\Vert\nabla g\Vert_{\dot{H}^{\gamma-2}}\\
\lesssim\Vert f\Vert_{\dot{H}^{\gamma}}+\Vert g\Vert_{\dot{H}^{\gamma-1}}.\label{eq:2V5}
\end{multline}
Collecting (\ref{eq:2v1}), (\ref{eq:2v=0000E9}), (\ref{eq:2v3}),
(\ref{eq:2v4}) and (\ref{eq:2V5}) we conclude that
\begin{equation}
\Vert(1-\chi_{\text{obst}})u\Vert_{L^{p}L^{q}}\lesssim\Vert f\Vert_{\dot{H}^{\gamma}}+\Vert g\Vert_{\dot{H}^{\gamma-1}}.\label{eq:21s3}
\end{equation}

\subsubsection{The second term: away from the trapped ray and near the obstacles}

Let us now deal with
\[
w:=\chi_{\text{obst}}(1-\chi_{\text{ray}})u.
\]
We denote here $\chi:=\chi_{\text{obst}}(1-\chi_{\text{ray}})$ and
onsider $\varphi\in C_{0}^{\infty}((-1,1))$ satisfying $\varphi\geq0$,
$\varphi(0)=1$ and $\sum_{j\in\mathbb{Z}}\varphi(s-j)=1$. We decompose
\[
\chi u=\sum_{j\in\mathbb{Z}}\varphi(t-j)\chi u=:\sum_{j\in\mathbb{Z}}u_{j}.
\]
Because $\chi$ is supported away from the trapped ray, using the
$L^{2}$ estimate away from the trapped ray (\ref{eq:globalaway})
combined with the local Strichartz estimates in time $1$ for each
$u_{j}$ allows us to recover the estimate in the full space, with
the exact same proof as \cite{MR2001179}, the only difference been
using (\ref{eq:globalaway}) instead of his the $L^{2}$ estimate
for non trapping geometries and we get:
\begin{equation}
\Vert\chi_{\text{obst}}(1-\chi_{\text{ray}})u\Vert_{L^{p}(\mathbb{R},L^{q})}\lesssim\Vert f\Vert_{\dot{H}^{\gamma}}+\Vert g\Vert_{\dot{H}^{\gamma-1}}.\label{eq:v2s3}
\end{equation}

\subsubsection{The third term: near the trapped ray}

We will denote here $\chi=\chi_{\text{obst}}\chi_{\text{ray}}$. We
will cut $u$ in time intervals of lenght $|\log h|$. Consider $\varphi\in C_{0}^{\infty}((-1,1))$
satisfying $\varphi\geq0$, $\varphi(0)=1$ and $\sum_{j\in\mathbb{Z}}\varphi(s-j)=1$.
We decompose
\[
\chi u=\sum_{j\in\mathbb{Z}}\varphi(\frac{t}{|\log h|}-j)\chi u=:\sum_{j\in\mathbb{Z}}u_{j}.
\]
The $u_{j}$ satisfy the equation
\[
(\partial_{t}^{2}-\Delta)u_{j}=F_{j}+G_{j}
\]
where
\begin{gather}
F_{j}=|\log h|^{-2}\varphi''(\frac{t}{|\log h|}-j)\chi u+2|\log h|^{-1}\varphi'(\frac{t}{|\log h|}-j)\chi\partial_{t}u,\label{eq:Fj}\\
G_{j}=-\varphi(\frac{t}{|\log h|}-j)[\Delta,\chi]u.\label{eq:Gj}
\end{gather}
We denote
\begin{align*}
v_{j}(t) & =\int_{(j-1)|\log h|}^{t}\frac{\sin(t-s)\sqrt{-\Delta}}{\sqrt{-\Delta}}F_{j}(s)ds,\\
w_{j}(t) & =\int_{(j-1)|\log h|}^{t}\frac{\sin(t-s)\sqrt{-\Delta}}{\sqrt{-\Delta}}G_{j}(s)ds,
\end{align*}
in such a way that $u_{j}=v_{j}+w_{j}$. By the $L^{2}$-global integrability
estimate near the trapped ray (\ref{eq:globalnear}) and (\ref{eq:Fj})
we get
\[
\sum_{j\in\mathbb{Z}}\Vert|\log h|F_{j}\Vert_{L^{2}(\mathbb{R},\dot{H}^{\gamma-1})}^{2}\lesssim|\log h|\left(\Vert u_{0}\Vert_{\dot{H}^{\gamma}}^{2}+\Vert u_{1}\Vert_{\dot{H}^{\gamma-1}}^{2}\right),
\]
and therefore
\begin{equation}
\sum_{j\in\mathbb{Z}}\Vert F_{j}\Vert_{L^{2}(\mathbb{R},\dot{H}^{\gamma-1})}^{2}\lesssim|\log h|^{-1}\left(\Vert u_{0}\Vert_{\dot{H}^{\gamma}}^{2}+\Vert u_{1}\Vert_{\dot{H}^{\gamma-1}}^{2}\right).\label{eq:bound_sumFj}
\end{equation}
On the other hand, by the Strichartz estimate on logarithmic interval
\begin{equation}
\Vert v_{j}\Vert_{L^{p}L^{q}}\lesssim\Vert F_{j}\Vert_{L^{1}\dot{H}^{\gamma-1}}.\label{eq:v_jinf}
\end{equation}
But, as $F_{j}$ is support on a time interval of size proportional
to $|\log h|$, by the Cauchy-Schwarz inequality we get
\begin{equation}
\Vert F_{j}\Vert_{L^{1}\dot{H}^{\gamma-1}}\lesssim|\log h|^{1/2}\Vert F_{j}\Vert_{L^{2}\dot{H}^{\gamma-1}}.\label{eq:Fj_inf}
\end{equation}
Therefore, by (\ref{eq:bound_sumFj}), (\ref{eq:v_jinf}) and (\ref{eq:Fj_inf})
\begin{equation}
\sum_{j\in\mathbb{Z}}\Vert v_{j}\Vert_{L^{p}L^{q}}^{2}\lesssim\left(\Vert u_{0}\Vert_{\dot{H}^{\gamma}}^{2}+\Vert u_{1}\Vert_{\dot{H}^{\gamma-1}}^{2}\right).\label{eq:sumvjfin}
\end{equation}

Now, let us deal with $w_{j}$. Let us define
\[
\tilde{w}_{j}^{\pm}=e^{-it\sqrt{-\Delta}}\int_{(j-1)|\log h|}^{(j+1)|\log h|}\frac{e^{is\sqrt{-\Delta}}}{\sqrt{-\Delta}}G_{j}(s)ds.
\]
Decomposing the sinus operator in half wave operators and make use
of the Christ-Kiselev lemma allows us to estimate the norm of $\tilde{w}_{j}^{\pm}$
instead of these of $w_{j}$. By the Strichartz estimates on logarithmic
interval we get
\[
\Vert\tilde{w}_{j}^{\pm}\Vert_{L^{p}L^{q}}\leq\Vert\int_{(j-1)|\log h|}^{(j+1)|\log h|}e^{is\sqrt{-\Delta}}G_{j}(s)ds\Vert_{\dot{H}^{\gamma-1}}.
\]
Now, remark that $[\Delta,\chi]$ is supported away from the periodic
way. Let $\tilde{\chi}$ be equal to $1$ in the support of $\nabla\chi$
and vanishing on the trapped ray. By the dual version of the $L^{2}$-global
integrability estimate outside the trapped ray (\ref{eq:globalaway})
we get
\begin{multline*}
\Vert\tilde{w}_{j}^{\pm}\Vert_{L^{p}L^{q}}\leq\Vert\int_{(j-1)|\log h|}^{(j+1)|\log h|}e^{is\sqrt{-\Delta}}G_{j}(s)ds\Vert_{\dot{H}^{\gamma-1}}\\
=\Vert\int_{(j-1)|\log h|}^{(j+1)|\log h|}e^{is\sqrt{-\Delta}}\tilde{\chi}G_{j}(s)ds\Vert_{\dot{H}^{\gamma-1}}\lesssim\Vert G_{j}\Vert_{L^{2}\dot{H}^{\gamma-1}}.
\end{multline*}
And now, by the $L^{2}$-global integrability estimate outside the
trapped ray (\ref{eq:globalaway}) itself:
\[
\sum_{j\in\mathbb{Z}}\Vert G_{j}\Vert_{L^{2}(\mathbb{R},\dot{H}^{\gamma-1})}^{2}\lesssim\left(\Vert u_{0}\Vert_{\dot{H}^{\gamma}}^{2}+\Vert u_{1}\Vert_{\dot{H}^{\gamma-1}}^{2}\right).
\]
Thereore we get: 
\begin{equation}
\sum_{j\in\mathbb{Z}}\Vert w_{j}\Vert_{L^{p}L^{q}}^{2}\lesssim\left(\Vert u_{0}\Vert_{\dot{H}^{\gamma}}^{2}+\Vert u_{1}\Vert_{\dot{H}^{\gamma-1}}^{2}\right).\label{eq:wj_sumfin}
\end{equation}

Thus, combining (\ref{eq:sumvjfin}) and (\ref{eq:wj_sumfin}) we
conclude thanks to the embedding $l^{2}(\mathbb{Z})\hookrightarrow l^{p}(\mathbb{Z})$
(we recall that $p\geq2$) :
\begin{multline*}
\Vert\chi u\Vert_{L^{p}L^{q}}\sim\left(\sum_{j\in\mathbb{Z}}\Vert u_{j}\Vert_{L^{p}L^{q}}^{p}\right)^{\frac{1}{p}}\lesssim\left(\sum_{j\in\mathbb{Z}}\Vert u_{j}\Vert_{L^{p}L^{q}}^{2}\right)^{\frac{1}{2}}\lesssim\Vert u_{0}\Vert_{\dot{H}^{\gamma}}+\Vert u_{1}\Vert_{\dot{H}^{\gamma-1}}.
\end{multline*}
Combining this last estimate with \eqref{21s3} and \eqref{v2s3}
we conclude that
\[
\Vert u\Vert_{L^{p}L^{q}}\lesssim\Vert u_{0}\Vert_{\dot{H}^{\gamma}}+\Vert u_{1}\Vert_{\dot{H}^{\gamma-1}}.
\]
Finally, standard technics permit to remove the frequencies cut-off.
Therefore Proposition \propref{red1} implies our main theorem.

\subsection{Reduction to the trapped set}

Let $D$ be an open-neighborhood of the trapped ray, choosen to be
a cylinder with the trapped ray for axis. We define the trapped set
of $D$ in time $T$
\begin{defn}
$(x,\xi)\in T^{\star}\Omega$ belongs to the trapped set of $D$ in
time $T$, denoted $\mathcal{T}_{T}(D)$, if and only if one of the
the ray starting from $(x,\frac{\xi}{|\xi|})$ and $(x,-\frac{\xi}{|\xi|})$
belongs to $D$ after time $T$. 

Note that the only differences with the definition of \cite{Schreodinger}
are that the rays are all followed at speed one instead of $|\xi|\in[\alpha_{0},\beta_{0}]$.
Therefore, with the same proofs, we get
\end{defn}
\begin{lem}
\label{lem:distbic}For all bicharacteristic $\gamma$ starting from
$D$ with speed one, we have
\[
d(\gamma(t),\mathcal{T}_{T}(D)^{c})>0\ \forall t\in[-T-1,-T]
\]
\end{lem}
and
\begin{lem}
\label{lem:distsup}For all $D,\tilde{D}$, there exists $T^{\star}>0$,
$c>0$ such that for all $T\geq0$:
\begin{equation}
d(\mathcal{T}_{T-T^{\star}}(D)^{c},\mathcal{T}_{T}(D))\geq e^{-cT},\label{eq:distsupT}
\end{equation}
and, if $D\subset\tilde{D}$
\begin{equation}
d(\mathcal{T}_{T}(\tilde{D})^{c},\mathcal{T}_{T}(D))\geq\frac{1}{4}e^{-cT}d(\tilde{D}{}^{c},D).\label{eq:distsupD}
\end{equation}
\end{lem}
We say that $f\in L^{2}$ is microlocally supported in $U\subset T^{\star}\Omega$,
if for all $a\in C^{\infty}(T^{\star}\Omega)$ such that $a=1$ in
$U$ we have $\text{Op}(a)f=f$. Using the same time translations
as in \cite{Schreodinger} combined with the finite speed of propagation,
the following Proposition implies our main theorem:
\begin{prop}
\label{redT}There exists $\epsilon>0$ and a small neighborhood $D$
of the trapped ray, such that, for all $\chi\in C_{0}^{\infty}$ supported
in $D$, if $f$, $g$ are such that $\psi(-h^{2}\Delta)f=f$, $\psi(-h^{2}\Delta)g=g$,
are microlocally supported in $T_{\epsilon|\log h|}(D)$ and spatially
in $D$ and away from $\partial(\Theta_{1}\cup\Theta_{2})$, and $u$
is the solution of (\ref{eq:lw}) with data $(f,g)$, we have:
\[
\Vert\chi u\Vert_{L^{p}(0,\epsilon|\log h|)L^{q}}\lesssim\Vert f\Vert_{\dot{H}^{\gamma}}+\Vert g\Vert_{\dot{H}^{\gamma-1}}.
\]
\end{prop}
The rest of the paper is thus devoted to prove Proposition \ref{redT}.

\section{Construction of an approximate solution}

\subsection{The phase functions}

We recall here the definition of the phase functions we used in \cite{Schreodinger}
following the works of Iwaka \cite{Ikawa2,IkawaMult} and Burq \cite{plaques}.
We call $\varphi:\mathcal{U}\rightarrow\mathbb{R}$ a phase function
on the open set $\mathcal{U}\subset\mathbb{R}^{3}$ if $\varphi$
is $C^{\infty}$ on $\mathcal{U}$ and verifies $|\nabla\varphi|=1$.
We say that $\varphi$ verifies $(P)$ on $\partial\Theta_{p}$ if
\begin{enumerate}
\item The principal curvatures of the level surfaces of $\varphi$ with
respect to $-\nabla\varphi$ are non-negative in every point of $\mathcal{U}$,
\item We have, for $j\neq p$
\[
\Theta_{j}\subset\{y+\tau\nabla\varphi(x)\ \text{s.t.}\ \tau\geq0,y\in\mathcal{U}\cap\partial\Theta_{p},\nabla\varphi(y)\cdot n(y)\geq0\},
\]
\item For all $A\in\mathbb{R}$, the set $\{\varphi\leq A\}$ is empty or
convex.
\end{enumerate}
Let $\delta_{1}\geq0$ and $\varphi$ be a phase function. We set

\begin{align*}
\Gamma_{p}(\varphi) & =\{x\in\partial\Theta_{p}\ \text{s.t.}\ -n(x)\cdot\nabla\varphi(x)\geq\delta_{1}\},\\
\mathcal{U}_{p}(\varphi) & =\underset{X^{1}(x,\nabla\varphi(x))\in\Gamma_{p}(\varphi)}{\bigcup}\{X^{1}(x,\nabla\varphi(x))+\tau\Xi(x,\nabla\varphi(x)),\ \tau\geq0\}.
\end{align*}
Then, there exists $\delta_{1}\geq0$ such that, if $\varphi$ is
a phase function verifying $(P\text{)}$ on $\partial\Theta_{p}$,
we can define the phase $\varphi_{j}$ reflected on the obstacle $\Theta_{j}$
on the open set $\mathcal{U}_{j}(\varphi)$, verifying $(P)$ on $\partial\Theta_{j}$,
by the following relation, for $X^{1}(x,\nabla\varphi(x))\in\Gamma_{p}(\varphi)$:
\[
\varphi_{j}(X^{1}(x,\nabla\varphi)+\tau\Xi^{1}(x,\nabla\varphi))=\varphi(X^{1}(x,\nabla\varphi))+\tau.
\]

We call a finite sequence $J=(j_{1},\cdots,j_{n})$, $j_{i}\in\{1,2\text{\}}$
with $j_{i}\neq j_{i+1}$ a story of reflections, and will denote
$\mathcal{I}$ the set of all the stories of reflection. By induction,
we can define the phases $\varphi_{J}$ for any $J\in\mathcal{I}$,
on the sets $\mathcal{U}_{J}(\varphi)$.

For $f\in C^{\infty}(\mathcal{U})$ and $m\in\mathbb{N}$, let 
\[
|f|_{m}(\mathcal{U})=\max_{(a_{i})\in(\mathcal{S}^{2})^{m}}\sup_{\mathcal{U}}|(a_{1}\cdot\nabla)\cdots(a_{m}\cdot\nabla)f|.
\]
The following estimate due to \cite{Ikawa2,IkawaMult,plaques}:
\begin{prop}
\label{prop:contrder}For every $m\geq0$ we have 
\[
|\nabla\varphi_{J}|_{m}\leq C_{m}|\nabla\varphi|_{m}.
\]
\end{prop}
Moreover, according to \cite{plaques}:
\begin{prop}
\label{prop:opens}There exists $M>0$ such that, for each $(i,j)\in\{1,2\}^{2}$,
there exists open sets containing the trapped ray $\mathcal{U}_{i,j}$
such that, if $J=\{i,\cdots,j\}$ verifies $|J|\geq M$, and $\varphi$
verifies $(P)$, $\varphi_{J}$ can be defined in $\mathcal{U}_{i,j}$. 
\end{prop}
We set
\[
\hat{\mathcal{U}}_{\infty}=\mathcal{U}_{11}\cap\mathcal{U}_{12}\cap\mathcal{U}_{21}\cap\mathcal{U}_{22},
\]
and $\mathcal{U}_{\infty}\subset\hat{\mathcal{U}}_{\infty}$ to be
an open cylinder having for axis the periodic trajectory and contained
in $\hat{\mathcal{U}}_{\infty}$. It will be shrinked in the sequel
if necessary. Finally, we recall the following estimate concerning
the derivatives with respect to $\xi$ of the phases builded begining
with $\varphi=(x-y)\cdot\frac{\xi}{|\xi|}$ we obtained in \cite{Schreodinger}:
\begin{prop}
\label{prop:decder}Let $\varphi(x)=(x-y)\cdot\frac{\xi}{|\xi|}$.
We denote $\varphi_{J}(x,\xi)$ the reflected phase we build begining
with $\varphi$. Then, for all multi-indices $\alpha,\beta$ there
exists a constant $D_{\alpha,\beta}>0$ such that the following estimate
holds on $\mathcal{U}_{\infty}$:
\[
|D_{\xi}^{\alpha}D_{x}^{\beta}\nabla\varphi_{J}|\leq D_{\alpha,\beta}^{|J|}.
\]
\end{prop}

\subsection{The microlocal cut-off}

According to Section 2, we are reduced to show Proposition \ref{redT}.
By \lemref{distsup}, we can construct a small shrinking of $\mathcal{U}_{\infty}$,
$\tilde{\mathcal{U}}_{\infty}\subset\mathcal{U}_{\infty}$ , and $\tilde{q}_{\epsilon,h}\in C^{\infty}(T^{\star}\Omega)$
such that $\tilde{q}_{\epsilon,h}=1$ in an open neighborhood of $\mathcal{T}_{2\epsilon|\log h|}(\mathcal{\tilde{U}}_{\infty})$,
$\tilde{q}_{\epsilon,h}=0$ outside $\mathcal{T}_{2\epsilon|\log h|}(\mathcal{U}_{\infty})$
in such a way that, for all multi-indice $\alpha$,
\begin{equation}
|\partial_{\alpha}\tilde{q}_{\epsilon,h}|\lesssim h^{-2|\alpha|c\epsilon}.\label{eq:contrq}
\end{equation}
It suffices to show Strichartz estimates in time $\epsilon|\log h|$
for data microlocally supported in $\mathcal{T}_{\epsilon|\log h|}(\mathcal{\tilde{U}}_{\infty})$
and spatially supported in $\mathcal{\tilde{U}}_{\infty}$ and away
from a small neighborhood $\mathcal{V}$ of $\partial\left(\Theta_{1}\cup\Theta_{2}\right)$.
Let $\chi_{0}\in C^{\infty}$ such that $\chi_{0}=0$ near $\partial\left(\Theta_{1}\cup\Theta_{2}\right)$
and $\chi_{0}=1$ outside $\mathcal{V}$. For such functions, $\chi_{0}\text{Op}(\tilde{q}_{\epsilon,h})f=f$,
thus it suffices to show
\begin{gather}
\Vert\chi e^{-it\sqrt{-\Delta}}\chi_{0}\text{Op}(\tilde{q}_{\epsilon,h})f\Vert_{L^{p}(0,\epsilon|\log h|)L^{q}}\lesssim\Vert f\Vert_{\dot{H}^{\gamma}},\label{eq:mlco2-1}
\end{gather}
for all $\chi\in C^{\infty}$ supported in $\mathcal{\tilde{U}}_{\infty}$.
We will show the strongest estimate:
\begin{equation}
\Vert e^{-it\sqrt{-\Delta}}\chi_{0}\text{Op}(\tilde{q}_{\epsilon,h})f\Vert_{L^{p}(0,\epsilon|\log h|)L^{q}}\lesssim\Vert f\Vert_{\dot{H}^{\gamma}}.\label{eq:mlco2}
\end{equation}
by the $TT^{\star}$ method - see for example \cite{KeelTao} - it
suffise to show the dispersive estimate, for $0\leq t\leq\epsilon|\log h|$:
\begin{equation}
\Vert Q_{\epsilon,h}^{\star}e^{-it\sqrt{-\Delta}}Q_{\epsilon,h}\Vert_{L^{1}\rightarrow L^{\infty}}\lesssim\frac{1}{h^{\frac{d+1}{2}}t^{\frac{d-1}{2}}}.\label{eq:disp-1}
\end{equation}
where
\[
Q_{\epsilon,h}:=\psi(-h^{2}\Delta)\chi_{0}\text{Op}(\tilde{q}_{\epsilon,h}).
\]

Then, to show (\ref{eq:disp-1}), it suffises to show
\begin{equation}
\Vert\text{Op}(q_{\epsilon,h,N})^{\star}e^{-it\sqrt{-\Delta}}\text{Op}(q_{\epsilon,h,N})\Vert_{L^{1}\rightarrow L^{\infty}}\lesssim\frac{1}{h^{\frac{d+1}{2}}t^{\frac{d-1}{2}}}.\label{eq:dispder}
\end{equation}
for $N$ large enough. Note that, in particular, 
\begin{equation}
\text{Supp}q_{\epsilon,h,N}\subset\mathcal{T}_{2\epsilon|\log h|}(\mathcal{U}_{\infty})\cap\mathcal{U}_{\infty}\times\{|\xi|\in[\alpha_{0},\beta_{0}]\}\label{eq:suppq}
\end{equation}
and $q_{\epsilon,h,N}$ is spatially supported outside a small neighborhood
of $\partial\left(\Theta_{1}\cup\Theta_{2}\right)$ not depending
of $\epsilon,h,N$.

It suffices to obtain
\begin{gather*}
\Vert\text{Op}(q_{\epsilon,h,N})^{\star}\cos(t\sqrt{-\Delta})\text{Op}(q_{\epsilon,h,N})\Vert_{L^{1}\rightarrow L^{\infty}}\lesssim\frac{1}{h^{\frac{d+1}{2}}t^{\frac{d-1}{2}}},\\
\Vert\text{Op}(q_{\epsilon,h,N})^{\star}\sin(t\sqrt{-\Delta})\text{Op}(q_{\epsilon,h,N})\Vert_{L^{1}\rightarrow L^{\infty}}\lesssim\frac{1}{h^{\frac{d+1}{2}}t^{\frac{d-1}{2}}}.
\end{gather*}
We will deal for example with the cosinus part, the sinus is handled
in the same way. We set

\[
\delta_{\epsilon,h,N}^{y}(x)=\frac{1}{(2\pi h)^{d}}\int e^{-i(x-y)\cdot\xi/h}q_{\epsilon,T,N}(x,\xi)d\xi,
\]
in order to have, for $u\in L^{2}$
\[
\left(\text{Op}(q_{\epsilon,h,N})u\right)(x)=\int\delta_{\epsilon,h,N}^{y}(x)u(y)dy.
\]
Notice that
\[
\text{Op}(q_{\epsilon,h,N})^{\star}\cos(t\sqrt{-\Delta})\text{Op}(q_{\epsilon,h,N})u(x)=\int\text{Op}(q_{\epsilon,h,N})^{\star}\cos(t\sqrt{-\Delta})\delta_{\epsilon,T,N}^{y}(x)u(y)dy,
\]
thus, to show (\ref{eq:dispder}), it suffises to study $\delta_{\epsilon,h,N}^{y}$
and to show that, for $N$ large enough
\[
|\text{Op}(q_{\epsilon,h,N})^{\star}\cos(t\sqrt{-\Delta})\delta_{\epsilon,h,N}^{y}|\lesssim\frac{1}{h^{\frac{d+1}{2}}t^{\frac{d-1}{2}}},\ \text{for }0\leq t\leq\epsilon|\log h|.
\]
Let $\mathcal{V}_{1}$ be a small neighborhood of $\partial(\Theta_{1}\cup\Theta_{2})$
on wich $q_{\epsilon,h,N}$ is vanishing and $\chi_{0}\in C_{0}^{\infty}(\mathbb{R}^{n})$
be such that $\chi_{0}=1$ on $\mathcal{U}_{\infty}\cap\mathcal{V}_{1}^{c}$.
We choose $\chi_{+}$ to be supported on $\text{Conv}(\Theta_{1}\cup\Theta_{2})\backslash(\Theta_{1}\cup\Theta_{2})$
and away from a small enough neighborhood of $\partial(\Theta_{1}\cup\Theta_{2})$,
$\text{Conv}$ denoting the convex hull. Note that in particular,
$\text{Op}(q_{\epsilon,h,N})^{\star}=\text{Op}(q_{\epsilon,h,N})^{\star}\chi_{+}$.
The symbol of $\text{Op}(q_{\epsilon,T,N})^{\star}$ enjoys the development
\[
q_{\epsilon,h,N}{}^{\star}(x,\xi)=e^{ih\langle D_{x},D_{\xi}\rangle}q_{\epsilon,h,N}.
\]
Thus, by (\ref{eq:contrq}), taking $\epsilon>0$ small enough, we
have $|q_{\epsilon,T,N}^{\star(\alpha)}|\lesssim1$ for all $|\alpha|\leq n+1=4$.
Moreover, $q_{\epsilon,T,N}^{\star(\alpha)}$ is compactly supported
in frequencies. Therefore, by \cite{boundedlp}, Section 4, $\text{Op}(q_{\epsilon,T,N})$
is bounded on $L^{\infty}\rightarrow L^{\infty}$ independently of
$h$. Therefore, we only have to show, for all $0\leq T\leq\epsilon|\log h|$
\begin{equation}
|\chi_{0}\cos(t\sqrt{-\Delta})\delta_{\epsilon,h,N}^{y}|\lesssim\frac{1}{h^{\frac{d+1}{2}}t^{\frac{d-1}{2}}},\ \text{for }0\leq t\leq\epsilon|\log h|\label{eq:ultimred}
\end{equation}
for $N$ large enough. 

In order to do so, we will construct a parametrix in time $0\leq t\leq\epsilon|\log h|$
for the wave equation with data $(\delta_{\epsilon,h,N}^{y},0)$.
The first step will be to construct an approximate solution of wave
equation with data 
\[
(e^{-i(x-y)\cdot\xi/h}q_{\epsilon,h,N}(x,\xi),0)
\]
where $\xi\in\mathbb{R}^{n},\xi\in\text{Supp}q_{\epsilon,h,N}$ is
fixed and considered as a parameter. 

\subsection{Approximation of the solution}

\subsubsection{The Neumann sum}

We look for the solution $w$ of

\[
\begin{cases}
\partial_{t}^{2}w-\Delta w & =0\ \text{in }\Omega\\
w(t=0)(x) & =e^{-i(x-y)\cdot\xi/h}q(x,\xi)\\
\partial_{t}w(t=0)(x) & =0\\
w_{|\partial\Omega} & =0
\end{cases}
\]
as the Neumann serie \emph{
\[
w=\sum_{J\in\mathcal{I}}(-1)^{|J|}w^{J}
\]
}where
\begin{equation}
\begin{cases}
\partial_{t}^{2}w^{\emptyset}-\Delta w^{\emptyset} & =0\ \text{in }\mathbb{R}^{n}\\
w(t=0)(x) & =e^{-i(x-y)\cdot\xi/h}q(x,\xi)\\
\partial_{t}w(t=0)(x) & =0
\end{cases}\label{eq:wempt}
\end{equation}
 and, for $J\neq\emptyset$, $J=(j_{1},\cdots,j_{n})$, $J'=(j_{1},\cdots,j_{n-1})$
\begin{equation}
\begin{cases}
\partial_{t}^{2}w^{J}-\Delta w^{J} & =0\ \text{in }\mathbb{R}^{n}\backslash\Theta_{j_{n}}\\
w(t=0),\partial_{t}w(t=0) & =0\\
w_{|\partial\Theta_{j_{n}}}^{J} & =w_{|\partial\Theta_{j_{n}}}^{J'}.
\end{cases}\label{eq:wJ}
\end{equation}
Let us denote
\[
\varphi_{J}^{+}(x,\xi)=\varphi_{J}(x,\xi)\text{ and }\varphi_{J}^{-}(x,\xi)=\varphi_{J}(x,-\xi),
\]
that is, $\varphi_{J}^{+}$ is the reflected phase constructed with
$\varphi(x,\xi)=(x-y)\cdot\frac{\xi}{|\xi|}$ and $\varphi_{J}^{+}$
is the reflected phase constructed with $\varphi(x,-\xi)=-(x-y)\cdot\frac{\xi}{|\xi|}$.
We look for $w^{J}$ as the sum of the two series

\begin{gather*}
w^{J}=w^{J,+}+w^{J,-}\\
=\sum_{k\geq0}w_{k}^{J,+}(x,t)e^{-i(\varphi_{J}^{+}|\xi|-t|\xi|)/h}(-i\frac{h}{|\xi|})^{k}+\sum_{k\geq0}w_{k}^{J,-}(x,t)e^{i(\varphi_{J}^{-}|\xi|-t|\xi|)/h}(i\frac{h}{|\xi|})^{k}.
\end{gather*}
If $w_{0}^{J,\pm}$ solves the transport equations
\begin{gather*}
(2\partial_{t}+\nabla\varphi_{J}^{\pm}\cdot\nabla+\Delta\varphi_{J})w_{0}^{J,\pm}=0,
\end{gather*}
and for $k\geq1$
\begin{gather*}
(2\partial_{t}+\nabla\varphi_{J}^{\pm}\cdot\nabla+\Delta\varphi_{J})w_{k}^{J,\pm}=-\Box w_{k-1}^{J,\pm},
\end{gather*}
with, for $J=\emptyset$
\begin{gather*}
w_{0}^{\emptyset,+}(x,0)=w_{0}^{\emptyset,-}(x,0)=\frac{1}{2}q(x,\xi),\\
w_{k}^{\emptyset,\pm}(x,0)=0,\ \forall k\geq1,\\
\partial_{t}w_{k}^{\emptyset,+}(x,0)+\partial_{t}w_{k}^{\emptyset,-}(x,0)=0,\ \forall k\geq0,
\end{gather*}
and, for $|J|\geq1$
\begin{gather*}
w_{_{k}|\partial\Theta_{j_{n}}}^{J,\pm}=w_{_{k}|\partial\Theta_{j_{n}}}^{J',\pm},\\
\partial_{t}w_{k}^{J,\pm}(x,0)=w_{k}^{J,\pm}(x,0)=0,
\end{gather*}
then $w^{J}$ solves (\ref{eq:wempt}), (\ref{eq:wJ}).

Solving the transport equations for $J=\emptyset$ gives immediatly
\begin{align*}
w_{0}^{\emptyset,+} & =\frac{1}{2}q(x-t\frac{\xi}{|\xi|},\xi),\\
w_{0}^{\emptyset,-} & =\frac{1}{2}q(x+t\frac{\xi}{|\xi|},\xi),\\
w_{k}^{\emptyset,\pm} & =-\int_{0}^{t}\Box w_{k-1}^{\pm}(x\mp(s-t)\frac{\xi}{|\xi|},s)ds\quad k\geq1.
\end{align*}

\subsubsection{Reflections on the obstacles}

Now, we would like to reflect $w^{\emptyset,\pm}$ on the obstacle.
To this purpose, starting from the phases $\varphi(x,\xi)=\frac{(x-y)\cdot\xi}{|\xi|}$
and $\varphi(x,-\xi)=-\frac{(x-y)\cdot\xi}{|\xi|}$ we would like
to define the reflected phases as explained in the first subsection.

We decompose the set of the stories of reflections as 
\[
\mathcal{I}=\mathcal{I}_{1}\cup\mathcal{I}_{2}
\]
where $\mathcal{I}_{1}$ are all stories begining with a reflection
on $\Theta_{1}$, that is of the form $(1,\cdots)$, and $\mathcal{I}_{2}$
begining with a reflection on $\Theta_{2}$, that is of the form $(2,\cdots)$.
Let $e$ be a unit vector with the same direction as $\mathcal{R}$.
We take $e$ oriented from $\Theta_{1}$ to $\Theta_{2}$. For $\frac{\xi}{|\xi|}$
in a small enough neighborhood $V$ of $\{e,-e\}$ we have
\begin{enumerate}
\item if $\xi\cdot e>0$, then $\frac{(x-y)\cdot\xi}{|\xi|}$ verifies $(P)$
on $\Theta_{1}$ and $-\frac{(x-y)\cdot\xi}{|\xi|}$ verifies $(P)$
on $\Theta_{2}$ 
\item if $\xi\cdot e<0$, then $\frac{(x-y)\cdot\xi}{|\xi|}$ verifies $(P)$
on $\Theta_{2}$ and $-\frac{(x-y)\cdot\xi}{|\xi|}$ verifies $(P)$
on $\Theta_{1}$ 
\end{enumerate}
Remark that 
\begin{itemize}
\item In situation $(1)$, the support of $w^{\emptyset,+}$ never cross
$\Theta_{1}$ and the support of $w^{\emptyset,-}$ never cross $\Theta_{2}$
in any time, 
\item in situation $(2)$, the support of $w^{\emptyset,+}$ never cross
$\Theta_{2}$ and the support of $w^{\emptyset,-}$ never cross $\Theta_{1}$
in any time. 
\end{itemize}
We set
\begin{itemize}
\item In situation $(1)$: $(\mathcal{I}_{+},\mathcal{I}_{-}):=(\mathcal{I}_{2},\mathcal{I}_{1})$,
\item in situation $(2)$: $(\mathcal{I}_{+},\mathcal{I}_{-}):=(\mathcal{I}_{1},\mathcal{I}_{2})$.
\end{itemize}
Then, (\ref{eq:wJ}) is satisfied for $w^{\pm}$ for all $J\in\mathcal{I}_{\mp}$:
indeed, because the support of $w^{\emptyset,\pm}$ never cross $\Theta_{i_{\mp}}$,
we have for all time $0=w_{|\partial\Theta_{i_{\mp}}}^{\emptyset,\pm}=w_{|\partial\Theta_{i_{\mp}}}^{\left\{ i_{\mp}\right\} }$,
and so on. Thus, we are reduced to construct the $w^{J,\pm}$'s for
$J\in\mathcal{I}_{\pm}$. In the same way as in \cite{Schreodinger},
schrinking $\mathcal{U}_{\infty}$ if necessary, all the phases we
we will be dealing with are well defined according to the previous
remarks.

Then, in the exact same way as \cite{Schreodinger}, we solve the
transport equations along the rays:
\begin{prop}
\label{prop:solref}We denote by $\hat{X}_{-t}(x,\nabla\varphi_{J}^{\pm})$
the backward spatial component of the flow starting from $(x,\nabla\varphi_{J}^{\pm})$,
defined in the same way as $X_{-t}(x,\nabla\varphi_{J}^{\pm})$, at
the difference that we ignore the first obstacle encountered if it's
not $\Theta_{j_{n}},$ and we ignore the obstacles after $|J|$ reflections.
Moreover, for $J=(j_{1}=i_{1},\dots,j_{n})\in\mathcal{I}_{\pm}$,
denote by
\[
J(x,t,\xi)=\begin{cases}
(j_{1},\cdots,j_{k}) & \text{if }\hat{X}_{-t}(x,\nabla\varphi_{J}^{\pm})\text{ has been reflected \ensuremath{n-k} times,}\\
\emptyset & \text{if }\hat{X}_{-t}(x,\nabla\varphi_{J}^{\pm})\text{ has been reflected \ensuremath{n} times}.
\end{cases}
\]
 Then, the $w_{k}^{J}$'s are given by, for $t\geq0$ and $x\in\mathcal{U}_{J}(\varphi)$
\[
w_{0}^{J,\pm}(x,t)=\Lambda\varphi_{J}^{\pm}(x,\xi)q(\hat{X}_{-t}(x,\nabla\varphi_{J}^{\pm}),\xi)
\]
where
\[
\Lambda\varphi_{J}^{\pm}(x,\xi)=\frac{G\varphi_{J}^{\pm}(x)}{G\varphi_{J}^{\pm}(X^{-1}(x,\nabla\varphi_{J}^{\pm}))}\times\cdots\times\frac{G\varphi^{\pm}(X^{-|J|-1}(x,\nabla\varphi_{J}^{\pm}))}{G\varphi^{\pm}(X^{-|J|}(x,\nabla\varphi_{J}^{\pm}))},
\]
and, for $k\geq1$, and $x\in\mathcal{U}_{J}(\varphi)$
\[
w_{k}^{J,\pm}(x,t)=\int_{0}^{t}g_{\varphi_{J}}(x,(t-s),\xi)\Box w_{k-1}^{J(x,\xi,\pm(t-s))}(\hat{X}_{-(t-s)}(x,\nabla\varphi_{J}^{\pm}),s)ds
\]
where

\[
g_{\varphi_{J}}^{\pm}(x,\xi,t)=\frac{G\varphi_{J}^{\pm}(x)}{G\varphi_{J}^{\pm}(X^{-1}(x,\nabla\varphi_{J}^{\pm}))}\times\cdots\times\frac{G\varphi_{J(x,t,\xi)}^{\pm}(X^{-|J(x,t,\xi)|-1)}(x,\nabla\varphi_{J}^{\pm}))}{G\varphi_{J(x,t,\xi)}^{\pm}(\hat{X}_{-t}(x,\nabla\varphi_{J}^{\pm}))}.
\]
\end{prop}
With the exact same proofs as in \cite{Schreodinger}, subsection
4.3.2, following the rays at speed one instead of speed $|\xi|\in[\alpha_{0},\beta_{0}]$,
$w^{J,\pm}$ verify the following properties:
\begin{prop}
\label{propw}We have
\begin{enumerate}
\item $w_{k}^{J}(x,t)\neq0\implies(\hat{X}_{-t}(x,\nabla\varphi_{J}),\xi)\in\text{Supp}q.$
\item For $x\notin\mathcal{U}_{J}(\varphi)$ and $0\leq t\leq\epsilon|\log h|$,
$w_{k}^{J,\pm}(x,t)=0$,
\item there exists $c_{1},c_{2}>0$ such that for every $J\in\mathcal{I}$,
the support of $w_{k}^{J,\pm}$ is included in $\left\{ c_{1}|J|\leq t\right\} $
and which of $\chi_{0}w_{k}^{J,\pm}$ is included in $\left\{ c_{1}|J|\leq t\leq c_{2}(|J|+1)\right\} $,
\item in times $0\leq t\leq\epsilon|\log h|$, $\chi_{0}w_{k}^{J,\pm}$
is supported in $\mathcal{U}_{\infty}$.
\end{enumerate}
\end{prop}
Moreover, Proposition \propref{contrder} combined with Proposition
\propref{solref} and (\ref{eq:contrq}) gives immediatly:
\begin{prop}
The following bound hold on $\mathcal{U}_{\infty}$
\[
|D_{\xi}^{\alpha}w_{k}^{J,\pm}|\lesssim C_{\alpha}^{|J|}h^{-(2k+|\alpha|)c\epsilon}.
\]
\end{prop}

\subsection{Decay estimates}

We recall the principal result who permits to estimate the decay of
the reflected solutions, namely the convergence of the product of
the Gaussian curvatures $\Lambda\varphi_{J}$ obtained by \cite{Ikawa2,IkawaMult}
and \cite{plaques}. In the present framework of two obstacles, it
writes:
\begin{prop}
\label{prop:convL}Let $0<\lambda<1$ be the product of the two eigenvalues
lesser than one of the Poincaré map associated with the periodic trajectory.
Then, there exists $0<\alpha<1$, and for $I=(1,2)$ and $I=(2,1)$,
for every $l\in\{\{1\},\{2\},\emptyset\}$, there exists a $C^{\infty}$
function $a_{I,l}$ defined in $\mathcal{U}_{\infty}$, such that,
for all $J=(\underset{r\text{ times}}{\underbrace{I,\dots,I}},l)$,
we have
\[
\underset{\mathcal{U}_{\infty}}{\sup}|\Lambda\varphi_{J}-\lambda^{r}a_{I,l}|_{m}\leq C_{m}\lambda^{r}\alpha^{|J|}.
\]
\end{prop}
Combined with the explicit expressions of Proposition \propref{solref}
and (\ref{eq:contrq}), this result gives as in \cite{Schreodinger}
the following decay:
\begin{prop}
\label{Decays}We following bounds hold on $\mathcal{U}_{\infty}$:
\[
|w_{k}^{J,\pm}|_{m}\leq C_{k}\lambda^{|J|}h^{-(2k+m)c\epsilon}.
\]
Moreover, on the whole space, $|w_{k}^{J}|_{m}\leq C_{k}h^{-(2k+m)c\epsilon}.$
\end{prop}

\subsection{Critical points of the phase}

We need to study the critical points of the phase in order to be able
to perform a stationary phase argument on the solution we are building.
At the difference of \cite{Schreodinger}, the phases here stationate
in whole directions. Therefore, we will perform a stationary phase
on each sphere $\mathcal{S}^{n-1}(0,s)$. To this purpose, we need
\begin{prop}
\label{prop:non_deg_pha}Let us denote
\[
\mathcal{S}_{J}^{\pm}(x,\xi,t)=\varphi_{J}^{\pm}(x,\xi)|\xi|-t|\xi|.
\]
Then, there exists $\eta>0$ such that for all $|J|\geq1$,
\begin{gather}
d(X^{-|J|}(x,\nabla\varphi_{J}^{\pm}(x,\xi)),y)\leq\eta\text{\text{ and }}w^{J,\pm}(x,t,\xi)\neq0\nonumber \\
\implies D_{\xi}\mathcal{S}_{J}^{\pm}(x,\xi,t)\neq0.\label{eq:non_deg1-1}
\end{gather}
Moreover, as soon as $d(X^{-|J|}(x,\nabla\varphi_{J}(x,\xi)),y)\text{\ensuremath{\geq}}\eta$,
for all $s>0$ and $x$ there exists a unique $s_{J}(x,s)\in\mathcal{S}^{n-1}(0,s)$
such that, for all $t\geq0$
\begin{equation}
w^{J,\pm}(x,\cdot,s_{J}^{\pm}(x,s))\neq0\text{ and }D_{\mathcal{S}^{n-1}(0,s)}\mathcal{S}_{J}^{\pm}(x,s_{J}^{\pm}(x,s),t)=0.\label{eq:crit2}
\end{equation}
Furthermore, if $d(X^{-|J|}(x,\nabla\varphi_{J}(x,\xi)),y)\geq\eta$,
\begin{equation}
\det D_{\mathcal{S}^{n-1}(0,s)}^{2}\mathcal{S^{\pm}}_{J}(x,s_{J}^{\pm}(x,s),t)>c.\label{eq:crit3}
\end{equation}
\end{prop}
\begin{proof}
For the seek of lisibility, we denote $\mathcal{S}_{J}=\mathcal{S}_{J}^{+}$,
$\varphi_{J}:=\varphi_{J}^{+}$, $w^{J}=w^{J,+}$ and we make the
proof for the positive part of the wave, $w^{J,+}$: for $w^{J,-}$,
the proof is the same. 

In the same way as in \cite{Schreodinger}, we obtain, differentiating
$|\nabla\varphi_{J}(x,\xi)|\xi||^{2}=|\xi|^{2}$ with respect to $\xi$
and integrating the transport equation obtained along the rays up
to the first phase:
\begin{equation}
D_{\xi}\mathcal{S}_{J}^{\pm}(x,\xi,t)=X^{-|J|}(x,\nabla\varphi_{J}(x,\xi))-y-(t-l_{J}(x,\xi))\frac{\xi}{|\xi|}.\label{eq:derSJ}
\end{equation}
Note that, by Proposition (\ref{propw}), (1), $w^{J,\pm}(x,\xi,t)\neq0$
implies that, because $q$ is supported away of the boundary, for
$|J|\geq1$ 
\[
t-l_{J}(x,\xi)\geq\delta_{0}>0
\]
and thus, we get (\ref{eq:non_deg1-1}). Moreover, we deduce that
\begin{multline*}
D_{\mathcal{S}^{n-1}(0,s)}\mathcal{S}_{J}(x,\xi)\\
=X^{-|J|}(x,\nabla\varphi_{J}^{\pm}(x,\xi))-y-\left(\left(X^{-|J|}(x,\nabla\varphi_{J}(x,\xi))-y\right)\cdot\frac{\xi}{|\xi|}\right)\frac{\xi}{|\xi|}.
\end{multline*}
Thus, if $\xi$ is such that $D_{\mathcal{S}^{n-1}(0,s)}\mathcal{S}_{J}(x,\xi)=0$,
\[
X^{-|J|}(x,\nabla\varphi_{J}(x,\xi))-y\sslash\frac{\xi}{|\xi|},
\]
hence $\frac{\xi}{|\xi|}$ is a direction allowing reaching the point
$x$ from the point $y$ following the story of reflection $J$. Note
that there is a priori two such vectors on $\mathcal{S}^{n-1}(0,s)$:
one and its opposite, but because $w^{J,+}=0$ for $J\in\mathcal{I}_{-}$,
we have $w^{J}(x,t,\xi)=0$ for one of them. We thus get (\ref{eq:crit2}).
Note that the critical point $\xi$ such that $w^{J}(x,t,\xi)\neq0$
is the one verifying
\begin{equation}
\left(X^{-|J|}(x,\nabla\varphi_{J}(x,\xi))-y\right)\cdot\frac{\xi}{|\xi|}>0.\label{eq:x-ypo}
\end{equation}

The Hessian in $\mathbb{R}^{n}$ of $\mathcal{S}_{J}$ is derivated
like in \cite{Schreodinger}, differentiating $|\nabla\varphi_{J}(x,\xi)|\xi||^{2}=|\xi|^{2}$
with respect to $\xi$ and integrating the transport equation obtained
along the rays once again:

\begin{multline*}
D_{\xi}^{2}S_{J}^{\pm}(x,\xi,t)=\frac{l}{|\xi|}Id-(\frac{Id}{|\xi|}-\frac{\xi\xi^{t}}{|\xi|^{3}})t\\
-\sum_{k=1}^{3}\int_{0}^{\frac{l}{|\xi|}}D_{\xi}\partial_{x_{k}}\psi_{J^{(s)}}(X_{-s}(x,\nabla\varphi_{J}(x,\xi)),\xi)\left(D_{\xi}\partial_{x_{k}}\psi_{J^{(s)}}(X_{-s}(x,\nabla\varphi_{J}(x,\xi)),\xi)\right)^{t}.
\end{multline*}
We would like to deduce an expression of $D_{\mathcal{S}^{n-1}(0,s)}^{2}S_{J}(x,\xi,t)$
for $\xi=s_{J}(x,t)$. To this purpose, we recall that
\begin{lem}
\label{restrc_hess}Let $g:\mathbb{R}^{n}\rightarrow\mathbb{R}^{m}$
be a submersion in $0$ and $M:=g^{-1}(0)$. Moreover, let $f:\mathbb{R}^{n}\rightarrow\mathbb{R}$
and $F$ be its restriction to $M$. We suppose that $F$ has a critical
point in $a\in M$. Then, the Hessian of $F$ in $a$ is the restriction
of
\[
d^{2}f_{a}-\lambda\circ d^{2}g_{a}
\]
to $T_{a}M$, where $\lambda$ is the Lagrange multiplicator of $f$
with respect to $g$ in $a$, that is the unique linear form $\lambda\in\mathcal{L}(\mathbb{R}^{m},\mathbb{R})$
such that $df_{a}=\lambda\circ dg_{a}.$
\end{lem}
Here, we can take $g(\xi):=|\xi|^{2}-s^{2}$. Then $Dg(\xi)=2\xi.$
The Lagrange multiplicator of $\mathcal{S}_{J}(x,\cdot)$ with respect
to $g$ in $\xi:=s_{J}(x,s)$ is the unique $\lambda\in\mathbb{R}$
such that
\[
D_{\xi}\mathcal{S}_{J}(x,t,\xi)=2\lambda\xi.
\]
Therefore, according to (\ref{eq:derSJ}) 
\[
2\lambda=\left(X^{-|J|}(x,\nabla\varphi_{J}(x,\xi))-y\right)\cdot\frac{\xi}{|\xi|}-(t-l_{J}(x,\xi))\frac{1}{|\xi|}.
\]
On the other hand, by lemma \ref{restrc_hess}
\[
D_{\mathcal{S}^{n-1}(0,s)}^{2}S_{J}(x,\xi,t)=\left(D_{\xi}^{2}S_{J}(x,\xi,t)-2\lambda\text{Id}\right)_{|T_{\xi}S^{n-1}(0,s)}
\]
Thus, $D_{\mathcal{S}^{n-1}(0,s)}^{2}S_{J}(x,\xi,t)$ is the restriction
to $T_{\xi}\mathcal{S}^{n-1}(0,s)$ of

\begin{gather*}
-(X^{-|J|}(x,\nabla\varphi_{J}^{\pm}(x,\xi))-y)\cdot\frac{\xi}{|\xi|}Id+\frac{\xi\xi^{t}}{|\xi|^{3}}\\
-\sum_{k=1}^{3}\int_{0}^{\frac{l}{|\xi|}}D_{\xi}\partial_{x_{k}}\psi_{J^{(s)}}(X_{-s}(x,\nabla\varphi_{J}(x,\xi)),\xi)\left(D_{\xi}\partial_{x_{k}}\psi_{J^{(s)}}(X_{-s}(x,\nabla\varphi_{J}(x,\xi)),\xi)\right)^{t}.
\end{gather*}
But, the quadratic form $\xi\xi^{t}$ vanishes in $T_{\xi}\mathcal{S}^{n-1}$.
Therefore, $D_{\mathcal{S}^{n-1}(0,s)}^{2}S_{J}(x,\xi,t)$ is the
restriction to $T_{\xi}\mathcal{S}^{n-1}(0,s)$ of
\begin{gather*}
-(X^{-|J|}(x,\nabla\varphi_{J}(x,\xi))-y)\cdot\frac{\xi}{|\xi|}\text{Id}\\
-\sum_{k=1}^{3}\int_{0}^{\frac{l}{|\xi|}}D_{\xi}\partial_{x_{k}}\psi_{J^{(s)}}(X_{-s}(x,\nabla\varphi_{J}(x,\xi)),\xi)\left(D_{\xi}\partial_{x_{k}}\psi_{J^{(s)}}(X_{-s}(x,\nabla\varphi_{J}(x,\xi)),\xi)\right)^{t}.\\
.
\end{gather*}
And as the matrices
\[
D_{\xi}\partial_{x_{k}}\psi_{J^{(s)}}(X_{-s}(x,\nabla\varphi_{J}(x,\xi)),\xi)\left(D_{\xi}\partial_{x_{k}}\psi_{J^{(s)}}(X_{-s}(x,\nabla\varphi_{J}(x,\xi)),\xi)\right)^{t}
\]
are positives, and according to (\ref{eq:x-ypo}), we get the last
part of the statement.
\end{proof}

\section{Proof of the main result}

Let $K\geq0$. By the previous section, the function
\[
(x,t)\rightarrow\frac{1}{(2\pi h)^{d}}\sum_{\pm}\sum_{J\in\mathcal{I}}\int\sum_{k=0}^{K}w_{k}^{J,\pm}(x,t,\xi)e^{\mp i(\varphi_{J}^{\pm}(x,\xi)|\xi|-t|\xi|)/h}(\mp ih/|\xi|)^{k}d\xi
\]
satisfies the approximate equation
\[
\partial_{t}^{2}u-\Delta u=(\mp h)^{K}\frac{1}{(2\pi h)^{d}}\sum_{\pm}\sum_{J\in\mathcal{I}}\int\Box w_{K-1}^{J,\pm}(x,t,\xi)e^{\mp i(\varphi_{J}^{\pm}(x,\xi)|\xi|-t|\xi|)/h}|\xi|^{-K}d\xi
\]
with data 
\[
(u(0),\partial_{t}u(0))=(\delta_{\epsilon,h,N}^{y},0).
\]
By the Duhamel formula, the difference from the actual solution, that
is from $\cos(t\sqrt{-\Delta})\delta^{y}$, is bounded in $H^{m}$
norm by 
\[
C\times h^{K-d}\times|t|\sup_{s\in[0,t],\xi\in[\alpha_{0},\beta_{0}]}\sum_{\pm}\sum_{J\in\mathcal{I}}\Vert\Box w_{K-1}^{J,\pm}(\cdot,s,\xi)e^{\mp i(\varphi_{J}^{\pm}(\cdot,\xi)|\xi|-t|\xi|)/h}\Vert_{H^{m}}.
\]
 So, for $0\leq t\leq\epsilon|\log h|$
\begin{equation}
\left(\cos(t\sqrt{-\Delta})\delta^{y}\right)(x)=S_{K}(x,t)+R_{K}(x,t)\label{eq:sumdelta}
\end{equation}
with
\begin{equation}
S_{K}(x,t)=\frac{1}{(2\pi h)^{d}}\sum_{\pm}\sum_{J\in\mathcal{I}}\int\sum_{k=0}^{K}w_{k}^{J,\pm}(x,t,\xi)e^{\mp i(\varphi_{J}^{\pm}(x,\xi)|\xi|-t|\xi|)/h}(\mp ih/|\xi|)^{k}d\xi\label{eq:SK}
\end{equation}
and
\begin{equation}
\Vert R_{K}\Vert_{H^{m}}\lesssim|\log h|h^{K-d}\sup_{s,\xi}\sum_{\pm}\sum_{J\in\mathcal{I}}\Vert\Box w_{K-1}^{J,\pm}(\cdot,s,\xi)e^{\mp i(\varphi_{J}^{\pm}(\cdot,\xi)|\xi|-t|\xi|)/h}\Vert_{H^{m}}.\label{eq:rK1}
\end{equation}

\subsection*{The reminder}

We first deal with the reminder term $R_{K}$. Let us denote
\[
W_{K-1}^{J,\pm}(x,s,\xi)=\Box w_{K-1}^{J,\pm}(\cdot,s,\xi)e^{\mp i(\varphi_{J}^{\pm}(\cdot,\xi)|\xi|-t|\xi|)/h}
\]
 Notice that, by construction of the $w_{k}$'s, $w_{k}^{J}$ is supported
in a set of diameter $(C+t)$. Therefore, using Proposition \ref{Decays},
Proposition \propref{contrder} and the derivative of a product we
get:
\[
\sum_{J\in\mathcal{I}}\Vert\partial^{m}W_{K-1}^{J,\pm}\Vert_{L^{2}}\lesssim C_{K}(1+t)^{\frac{d}{2}}\sum_{J\in\mathcal{I}}\Vert\partial^{m}W_{K-1}^{J,\pm}\Vert_{L^{\infty}}\lesssim C_{K}(1+t)^{\frac{d}{2}}h^{-m}\sum_{J\in\mathcal{I}}h^{-(2K+m)c\epsilon}
\]
and thus, by the Sobolev embedding $H^{d}\hookrightarrow L^{\infty}$
and (\ref{eq:rK1})
\[
\Vert R_{K}\Vert_{L^{\infty}}\lesssim C_{K}|\log h|h^{K-2d}(1+t)^{\frac{d}{2}}h^{-(2K+d)c\epsilon}|\left\{ J\in\mathcal{I},\text{ s.t }w_{K-1}^{J}\neq0\right\} |.
\]
By Proposition \ref{propw} we get in the same way as in \cite{Schreodinger}
\begin{equation}
|\left\{ J\in\mathcal{I},\text{ s.t }w_{K-1}^{J}\neq0\right\} |\lesssim(1+t),\label{eq:wkJnonnul}
\end{equation}
and therefore
\begin{equation}
\Vert R_{K}\Vert_{L^{\infty}}\lesssim C_{K}|\log h|h^{K(1-2c\epsilon)-d(2+c\epsilon)}(1+t)^{\frac{d}{2}+1}\label{eq:estRK}
\end{equation}
Thus, for $0\leq t\leq\epsilon|\log h|$ 
\[
\Vert R_{K}\Vert_{L^{\infty}}\lesssim C_{K}|\log h|^{\frac{d}{2}+1}h^{K(1-2c\epsilon)-d(2+c\epsilon)}\leq C_{K}h^{K(1-2c\epsilon)-d(1+c\epsilon)-1}.
\]
We take $\epsilon>0$ small enough so that $2c\epsilon\leq\frac{1}{2}$
and we get
\[
\Vert R_{K}\Vert_{L^{\infty}}\leq C_{K}h^{\frac{K}{2}-3d-1}.
\]
Let us now fix $K$ large enough so that
\[
\frac{K}{2}-3d-1\geq-\frac{d+1}{2}+1.
\]
Then, as $0\leq t\leq\epsilon|\log h|$ is equivalent to $h\leq e^{-\frac{t}{\epsilon}}$,
we obtain
\begin{equation}
\Vert R_{K}\Vert_{L^{\infty}}\leq C_{K}h^{-\frac{d+1}{2}}e^{-\frac{t}{\epsilon}}.\label{eq:RK}
\end{equation}
for $0\leq t\leq\epsilon|\log h|$.

\subsection*{The free wave $J=\emptyset$}

Let us denote 
\[
S_{K}^{\emptyset}(x,t)=\frac{1}{(2\pi h)^{d}}\sum_{\pm}\int\sum_{k=0}^{K}(\mp ih/|\xi|)^{k}w_{k}^{\emptyset,\pm}(x,t,\xi)e^{-i((x-y)\cdot\xi\mp t|\xi|)/h}d\xi
\]
the free part of the wave, and
\[
S_{K}^{r}=\frac{1}{(2\pi h)^{d}}\sum_{\pm}\sum_{|J|\geq1}\int\sum_{k=0}^{K}w_{k}^{J,\pm}(x,t,\xi)e^{\mp i(\varphi_{J}^{\pm}(x,\xi)|\xi|-t|\xi|)/h}(\mp ih/|\xi|)^{k}d\xi,
\]
the reflected waves, in such a way that
\[
S_{K}=S_{K}^{\emptyset}+S_{K}^{r}.
\]
Note that $S_{K}^{\emptyset}$ is simply the approximate expression
of the solution of the wave equation with data $(\delta^{y},0)$,
in the free space:
\[
S_{K}^{\emptyset}(x,t)=\left(\cos(t\sqrt{-\Delta_{0}})\delta^{y}\right)(x)+R_{K}^{\emptyset}(x,t)
\]
where $\Delta_{0}$ denote the Laplacian in the free space and by
the Duhamel formula, for $0\leq t\leq\epsilon|\log h|$
\[
\Vert R_{K}^{\emptyset}\Vert_{H^{m}}\lesssim h^{K-d}|\log h|\sup_{s,\xi}\Vert\Box w^{\emptyset}(\cdot,s,\xi)\Vert_{H^{m}}
\]
The usual dispersive estimate for the waves in the free space gives,
by the frequencies localization of $\delta^{y}$
\[
|\cos(t\sqrt{-\Delta_{0}})\delta^{y}|\lesssim\frac{1}{h^{\frac{d+1}{2}}t^{\frac{d-1}{2}}}
\]
and thus dealing with $R_{K}^{\emptyset}$ as we did for $R_{K}$
we get
\begin{equation}
|S_{K}^{\emptyset}|\lesssim\frac{1}{h^{\frac{d+1}{2}}t^{\frac{d-1}{2}}},\text{ for }0\leq t\leq\epsilon|\log h|.\label{eq:SKsmallt}
\end{equation}

\subsection*{The reflected waves $|J|\geq1$}

According to Proposition \propref{non_deg_pha}, the parts 
\[
d(X^{-|J|}(x,\nabla\varphi_{J}^{\pm}),y)\leq\eta
\]
 enjoys a rapid decay and we thus have
\begin{multline*}
S_{K}^{r}=\frac{1}{(2\pi h)^{d}}\sum_{\pm}\sum_{|J|\geq1}\int\sum_{k=0}^{K}w_{k}^{J,\pm}(x,t,\xi)e^{\mp i(\varphi_{J}^{\pm}(x,\xi)|\xi|-t|\xi|)/h}1_{d(X^{-|J|}(x,\nabla\varphi_{J}^{\pm}),y)\geq\eta}(\mp ih/|\xi|)^{k}d\xi\\
+\sum_{|J|\geq1}O(h^{\infty}).
\end{multline*}
Note that, by (\ref{eq:wkJnonnul}), the $O(h^{\infty})$ part does
not contribute. We write the remaining part of $S_{K}^{r}$ as, illegetimately
omiting $1_{d(X^{-|J|}(x,\nabla\varphi_{J}^{\pm}),y)\geq\eta}$ for
the seek of lisibility:

\begin{multline*}
S_{K}^{r}(x,t)=\frac{1}{(2\pi h)^{d}}\sum_{\pm}\sum_{J\in\mathcal{I}}\int\sum_{k=0}^{K}w_{k}^{J,\pm}(x,t,\xi)e^{\mp i(\varphi_{J}^{\pm}(x,\xi)|\xi|-t|\xi|)/h}(\mp ih/|\xi|)^{k}d\xi\\
=\frac{1}{(2\pi h)^{d}}\sum_{\pm}\sum_{J\in\mathcal{I}}\int_{s=\alpha_{0}}^{\beta_{0}}\int_{|\xi|=s}\sum_{k=0}^{K}w_{k}^{J,\pm}(x,t,\xi)e^{\mp i(\varphi_{J}^{\pm}(x,\xi)|\xi|-t|\xi|)/h}(\mp ih/s)^{k}d\xi ds.
\end{multline*}
According to Proposition \propref{non_deg_pha}, we can perform a
stationnary phase on each sphere $\left\{ |\xi|=s\right\} $, for
each term of the sum $\sum_{\pm}\sum_{|J|\geq1}$, up to order $h^{k_{0}}$.
We obtain, as the sphere is of dimension $d-1$, for $t\geq0$
\begin{multline}
S_{K}^{r}(x,t)=\frac{1}{(2\pi h)^{d}}h^{\frac{d-1}{2}}\sum_{\pm,J\in\mathcal{I}}\sum_{k=0}^{k_{0}}\int_{\alpha_{0}}^{\beta_{0}}e^{\mp i(\varphi_{J}^{\pm}(x,s_{J}(t,x))s-ts)/h}(\mp ih/s)^{k}\tilde{w}_{k}^{J,\pm}(t,x)ds\\
+\frac{1}{(2\pi h)^{d}}h^{\frac{d-1}{2}}\sum_{\pm,J\in\mathcal{I}}\int_{\alpha_{0}}^{\beta_{0}}R_{\text{st.ph.}}^{J,\pm}(x,t,s)ds\\
+\frac{1}{(2\pi h)^{d}}\sum_{\pm,J\in\mathcal{I}}\int\sum_{k=k_{0}+1}^{K}w_{k}^{J}(x,t,\xi)e^{\mp i(\varphi_{J}^{\pm}(x,\xi)|\xi|-t|\xi|)/h}(\mp ih/|\xi|)^{k}d\xi.\label{eq:sk}
\end{multline}
where for $0\leq k\leq k_{0}$, the term $\tilde{w}_{k}^{J,\pm}$
is a linear combination of
\begin{gather*}
D_{\xi}^{2k}w_{0}^{J,\pm}(t,x,s_{J}(t,x,s)),D_{\xi}^{2(k-1)}w_{1}^{J,\pm}(t,x,s_{J}(t,x,s)),\dots,w_{k}^{J,\pm}(t,x,s_{J}(t,x,s)),
\end{gather*}
where $s_{J}(t,x,s)$ is the stationnary point of the phase on the
hypersphere $\{|\xi|=s\}$ and $R_{\text{st.ph.}}^{J,\pm}$ is the
reminder involved in the stationnary phase. Then the proof proceed
as in \cite{Schreodinger}, Section 5, and we obtain in the exact
same way, combining the decays estimates of Proposition \ref{Decays}
with the informations of the temporal support of $w^{J}$ given by
Proposition \ref{propw}, for $\epsilon>0$ small enough depending
only of $\alpha_{0},\beta_{0}$ and of the geometry of the obstacles,
for $0\leq t\leq\epsilon|\log h|$ and $x\in\text{Supp}\chi_{0}$
:
\begin{gather*}
\sum_{J\in\mathcal{I}}|w_{k}^{J,\pm}|\leq C_{k}h^{-\frac{1}{2}}e^{-\mu t},\ 1\leq k\leq K-1,\\
\sum_{J\in\mathcal{I}}|w_{0}^{J,\pm}|\lesssim e^{-\mu t},\\
\sum_{J\in\mathcal{I}}|D_{\xi}^{2k}w_{0}^{J,\pm}|\lesssim C_{k}h^{-\frac{1}{2}}e^{-t/4\epsilon},\ 1\leq k\leq k_{0},\\
\sum_{J\in\mathcal{I}}|R_{\text{st.ph.}}^{J,\pm}|\lesssim e^{-t/\epsilon}.
\end{gather*}
for some $\mu=\mu(\epsilon)$ . Therefore we obtain, taking $k_{0}=\lceil\frac{d-1}{2}\rceil$,
for some $\nu>0$ depending only of $\alpha_{0},\beta_{0}$ and of
the geometry of the obstacles
\begin{equation}
|\chi_{0}S_{K}^{r}(x,t)|\lesssim h^{-\frac{d+1}{2}}e^{-\nu t}\ \text{ for }0\leq t\leq\epsilon|\log h|.\label{eq:SKgrand}
\end{equation}

\subsection*{Conclusion}

Thus, collecting (\ref{eq:sumdelta}), (\ref{eq:RK}), (\ref{eq:SKsmallt})
and (\ref{eq:SKgrand}) we get
\[
|\chi_{0}\text{cos}(t\sqrt{-\Delta})\delta_{\epsilon,h,N}^{y}|\lesssim\frac{1}{h^{\frac{d+1}{2}}t^{\frac{d-1}{2}}},\text{ for }0\leq t\leq\epsilon|\log h|.
\]
That is (\ref{eq:ultimred}). Thus, \thmref{main} is proved by the
work of reduction of the previous sections.

\section{The non linear problem}

Let us now consider the following defocusing non linear wave equation
in $\mathbb{R}^{3}\backslash\mathcal{K}$
\[
\begin{cases}
\partial_{t}^{2}u-\Delta_{D}u+u^{5}=0\\
(u(0),\partial_{t}u(0))=(f,g).
\end{cases}
\]
$\mathcal{K}$ will be the reunion of two balls, or an illuminated
obstacle as defined in the introduction, and we are concerned by the
scattering problem in both situations. Our main tool will be the following
momentum identity, which was first introduced by Morawetz \cite{Morawetz61}
in a similar form to show some decay properties of the linear wave
equation:
\begin{lem}
\label{lem:moment}Let $u$ be a solution of (NLW) in $\Omega$ and
$\chi\in C^{\text{\ensuremath{\infty}}}(\Omega,\mathbb{R})$. Then
we have
\begin{multline}
\partial_{t}\left(\int_{\Omega}-\partial_{t}u\nabla u\nabla\chi-\frac{1}{2}\Delta\chi u\partial_{t}u\right)=\int_{\Omega}(D^{2}\chi\nabla u,\nabla u)-\frac{1}{4}\int_{\Omega}u^{2}\Delta^{2}\chi\\
+\frac{3}{2}\int_{\Omega}|u|^{6}\Delta\chi-\frac{1}{2}\int_{\partial\Omega}|\partial_{n}u|^{2}\partial_{n}\chi.\label{eq:mor}
\end{multline}
\end{lem}
\begin{proof}
The identity can be shown by standard integrations by parts justified
by a limiting argument.
\end{proof}

\subsection{A scattering criterion}

The scattering in $\mathbb{R}^{3}$ was shown by Bahouri and Shatah
\cite{BahSha}. Their proof still hold in the case of a domain with
boundaries if we are able to control the boundary term appearing in
their computations, that is
\begin{lem}
\label{lem:bordBS}Let $u$ be a solution of (NLW) in a finite-border
domain $\Omega$ of $\mathbb{R}^{3}$ such that Strichartz estimates
(\ref{eq:striuch}) holds. If
\[
\frac{1}{T}\int_{0}^{T}\int_{\partial\Omega}|\partial_{n}u|^{2}d\sigma dt\longrightarrow0,
\]
as $T$ goes to infinity, then $u$ scatter in $\dot{H}^{1}$.
\end{lem}
Note that the trace of the normal derivative is not an easy object
to deal with, because this trace is a priori not defined in $L^{2}(\partial\Omega)$
for elements of $\dot{H}^{1}(\Omega)$. Moreover, even if we can define
it for almost every $u(t)$ when $u$ is a solution of (NLW) because
of the particular structure of the equation, the application 
\[
u\in\dot{H}^{1}\cap\left\{ \text{value in time \ensuremath{t} of solutions of NLW}\right\} \longrightarrow\partial_{n}u\in L^{2}(\partial\Omega)
\]
 is in our knowledge not known to be continuous. 

For this reason, we prefer to deal with the following criterion, which
involve only the local energy of the equation, and that we deduce
from the previous one using the momentum identity (\ref{eq:mor}):
\begin{lem}
\label{lem:contrenerg}Let $u$ be a solution of (NLW) in a finite-border
domain $\Omega$ of $\mathbb{R}^{3}$ such that Strichartz estimates
(\ref{eq:striuch}) holds. There exists $A>0$, $B(0,A)\supset\partial\Omega$,
such that, if
\begin{equation}
\frac{1}{T}\int_{0}^{T}\int_{\Omega\cap B(0,A)}|\nabla u(x,t)|^{2}+|u(x,t)|^{6}\ dxdt\longrightarrow0,\label{eq:contrenerg}
\end{equation}
as $T$ goes to infinity, then $u$ scatters in $\dot{H}^{1}$.
\end{lem}
\begin{proof}
Let $\chi\in C_{0}^{\infty}(\mathbb{R}^{3},\mathbb{R})$ be such that
$\nabla\chi=-n$ on $\partial\Omega$, supported in $B(0,A)$. Suppose
that
\[
\frac{1}{T}\int_{0}^{T}\int_{\Omega\cap B(0,A)}|\nabla u(x,t)|^{2}+|u(x,t)|^{6}\ dxdt\longrightarrow0
\]
as $T$ goes to infinity. We use \lemref{moment} with the weight
$\chi$ to get:
\begin{multline*}
\partial_{t}\left(\int_{\Omega}-\partial_{t}u\nabla u\nabla\chi-\frac{1}{2}\Delta\chi u\partial_{t}u\right)=\int_{\Omega}(D^{2}\chi\nabla u,\nabla u)-\frac{1}{4}\int_{\Omega}u^{2}\Delta^{2}\chi\\
+\frac{2}{3}\int_{\Omega}|u|^{6}\Delta\chi+\frac{1}{2}\int_{\partial\Omega}|\partial_{n}u|^{2}d\sigma.
\end{multline*}
Integrating in time we get
\begin{multline*}
\int_{0}^{T}\int_{\partial\Omega}|\partial_{n}u|^{2}d\sigma dt\lesssim\int_{\Omega\cap B(0,A)}|\partial_{t}u\nabla u|+|u\partial_{t}u|+\int_{0}^{T}\int_{\Omega\cap B(0,A)}|u|^{6}+|u|^{2}+|\nabla u|^{2},
\end{multline*}
and using Minkowsky inequality,
\begin{multline*}
\int_{0}^{T}\int_{\partial\Omega}|\partial_{n}u|^{2}d\sigma dt\lesssim\left(\int_{\Omega}|\partial_{t}u|^{2}\right)^{\frac{1}{2}}\left(\int_{\Omega}|\nabla u|^{2}\right)^{\frac{1}{2}}+A^{\frac{1}{3}}\left(\int_{\Omega}|\partial_{t}u|^{2}\right)^{\frac{1}{2}}\left(\int_{\Omega}|u|^{6}\right)^{\frac{1}{6}}\\
+\int_{0}^{T}\int_{\Omega\cap B(0,A)}\left(|u|^{6}+|\nabla u|^{2}\right)+A^{\frac{2}{3}}\int_{0}^{T}\left(\int_{\Omega\cap B(0,A)}|u|^{6}\right)^{\frac{1}{3}}\\
\lesssim_{A}C(E)+\int_{0}^{T}\int_{\Omega\cap B(0,A)}(|u|^{6}+|\nabla u|^{2})+T^{\frac{2}{3}}\left(\int_{0}^{T}\int_{\Omega\cap B(0,A)}|u|^{6}\right)^{\frac{1}{3}}.
\end{multline*}
 Thus
\begin{multline*}
\frac{1}{T}\int_{0}^{T}\int_{\partial\Omega}|\partial_{n}u|^{2}d\sigma dt\lesssim_{A}\frac{C(E)}{T}+\frac{1}{T}\int_{0}^{T}\int_{\Omega\cap B(0,A)}(|u|^{6}+|\nabla u|^{2})\\
+\left(\frac{1}{T}\int_{0}^{T}\int_{\Omega\cap B(0,A)}|u|^{6}\right)^{\frac{1}{3}}\longrightarrow0
\end{multline*}
as $T\rightarrow\infty$ and by \lemref{bordBS} we conclude that
$u$ scatter in $\dot{H}^{1}$.
\end{proof}
Notice that the Morawetz identity (\ref{eq:mor}) permits to obtain
this criterion as soon as one has a weight function $\chi$ such that
$\nabla\chi\cdot n\geq0$ on $\partial\Omega$, $D^{2}\chi$ is positive
definite, and $\Delta^{2}\chi\leq0$. Constructing such weights will
therefore be of key interest in the sequel.

\subsection{A partial result in the exterior of two balls}

In the exterior of two balls, \lemref{moment} seems not to be sufficient
to show the scattering criterion (\ref{eq:contrenerg}) because we
are not able to find an appropriate weight funtion. However, we can
choose a weight function which has the right behavior everywhere except
in a neighborhood of the trapped ray, and therefore obtain Theorem
\ref{th2}, which is a first step toward the scattering for all data:
it is extended to the exterior of two convex obstacles and used to
show the scattering in this framework in the work in progress \cite{LafLaurent}.
\begin{proof}[Proof of Theorem \ref{th2}.]
Without loss of generality, we suppose that $\Theta_{1}$ is centered
in $0$. We denote by $c$ the center of $\Theta_{2}$. We choose
the weight
\[
\chi(x):=|x|+|x-c|
\]
and use \lemref{moment} with weight $\chi$
\begin{multline*}
\partial_{t}\left(\int_{\Omega}-\partial_{t}u\nabla u\nabla\chi-\frac{1}{2}\Delta\chi u\partial_{t}u\right)=\int_{\Omega}(D^{2}\chi\nabla u,\nabla u)-\frac{1}{4}\int_{\Omega}u^{2}\Delta^{2}\chi\\
+\frac{2}{3}\int_{\Omega}|u|^{6}\Delta\chi-\frac{1}{2}\int_{\partial\Omega}|\partial_{n}u|^{2}\nabla\chi\cdot n\ d\sigma.
\end{multline*}
 Remark that $-\nabla\chi\cdot n\geq0$ on $\partial\Theta_{1}\cup\partial\Theta_{2}$:
indeed, on $\Theta_{1}$, $-n=\frac{x}{|x|}$ and thus
\[
-\nabla\chi\cdot n=1+\frac{x-c}{|x-c|}\cdot\frac{x}{|x|}\geq1-\left|\frac{x-c}{|x-c|}\right|\left|\frac{x}{|x|}\right|=0,
\]
 and the same hold on $\Theta_{2}$. Moreover, $\Delta^{2}\chi=0$.
Thus we obtain
\begin{equation}
\partial_{t}\left(-\int_{0}^{T}\partial_{t}u\nabla u\cdot\nabla\chi+\frac{1}{2}\Delta\chi u\partial_{t}u\right)\geq\frac{2}{3}\int_{\Omega}|u|^{6}\Delta\chi+\int_{\Omega}(D^{2}\chi\nabla u,\nabla u)\label{eq:morcontr}
\end{equation}
Integrating this inequality and controling the left-hand side using
the Hardy inequality 
\[
\int_{\Omega}\frac{|f|^{2}}{|x|^{2}}\lesssim\int_{\Omega}|\nabla f|^{2}\ \text{for }f\in\dot{H}_{0}^{1}(\Omega)
\]
 we get
\begin{equation}
\int_{0}^{T}\int_{\Omega}|u|^{6}\Delta\chi+(D^{2}\chi\nabla u,\nabla u)\ dxdt\lesssim E.\label{eq:eq1}
\end{equation}

From the one hand, $\Delta\chi\gtrsim\frac{1}{A}$ on $B(0,A)$, thus
\[
\int_{\Omega\cap B(0,A)}|u|^{6}\lesssim A\int_{\Omega\cap B(0,A)}|u|^{6}\Delta\chi\lesssim A\int_{\Omega}|u|^{6}\Delta\chi,
\]
and therefore, by (\ref{eq:eq1})
\begin{equation}
\frac{1}{T}\int_{0}^{T}\int_{\Omega\cap B(0,A)}|u|^{6}\ dxdt\lesssim\frac{E}{T}.\label{eq:decayl6}
\end{equation}

Now, we would like to estimate the localised cinetic energy using
(\ref{eq:eq1}) again. We have
\[
D^{2}\chi=\frac{1}{|x|}\left(\text{Id}-\frac{xx^{t}}{|x|^{2}}\right)+\frac{1}{|x-c|}\left(\text{Id}-\frac{(x-c)(x-c)^{t}}{|x-c|^{2}}\right).
\]
The operators corresponding to the matrices 
\[
\text{Id}-\frac{xx^{t}}{|x|^{2}},\text{ resp. }\text{Id}-\frac{(x-c)(x-c)^{t}}{|x-c|^{2}}
\]
are the orthogonal projections on the plane normal to $\frac{x}{|x|}$,
resp. to $\frac{x-c}{|x-c|}$. Thus, 
\begin{equation}
(D^{2}\chi\cdot\xi,\xi)=\left(\frac{1}{|x|}+\frac{1}{|x-c|}\right)|\xi|^{2}-\frac{1}{|x|}\left(\xi\cdot\frac{x}{|x|}\right)^{2}-\frac{1}{|x-c|}\left(\xi\cdot\frac{x-c}{|x-c|}\right)^{2}.\label{eq:d2xi1}
\end{equation}
We choose coordinates (depending of $x$ and $c$) such that 
\[
\frac{x}{|x|}=(1,0,0),\ \frac{x-c}{|x-c|}=(\cos\theta,\sin\theta,0),
\]
then we have, if $\xi=\begin{pmatrix}\hat{\xi}_{1} & \hat{\xi}_{2} & \hat{\xi}_{3}\end{pmatrix}$
in this set of coordinates
\[
\frac{1}{|x|}\left(\xi\cdot\frac{x}{|x|}\right)^{2}+\frac{1}{|x-c|}\left(\xi\cdot\frac{x-c}{|x-c|}\right)^{2}=\begin{pmatrix}\hat{\xi}_{1} & \hat{\xi}_{2}\end{pmatrix}\begin{pmatrix}\frac{1}{|x|}+\frac{\cos^{2}\theta}{|x-c|} & \frac{\sin\theta\cos\theta}{|x-c|}\\
\frac{\sin\theta\cos\theta}{|x-c|} & \frac{\sin^{2}\theta}{|x-c|}
\end{pmatrix}\begin{pmatrix}\hat{\xi}_{1}\\
\hat{\xi}_{2}
\end{pmatrix}.
\]
The largest eigenvalue of this positive quadratic form in $\begin{pmatrix}\hat{\xi}_{1} & \hat{\xi}_{2}\end{pmatrix}$
writes
\[
\lambda_{2}=\frac{1}{2}\left(\frac{1}{|x-c|}+\frac{1}{|x|}+\sqrt{\left(\frac{1}{|x-c|}+\frac{1}{|x|}\right)^{2}-4\frac{\sin^{2}\theta}{|x||x-c|}}\right)
\]
therefore, there exists $\alpha_{0}>0$ small enought and $c>0$ such
that, if $\alpha\leq\alpha_{0}$, we have, for $x\in\Omega\cap B(0,A)$
\begin{equation}
\sin^{2}\theta\geq\alpha\implies\lambda_{2}\leq\frac{1}{|x-c|}+\frac{1}{|x|}-c\alpha.\label{eq:d2xi2}
\end{equation}
On the other hand
\[
\frac{1}{|x|}\left(\xi\cdot\frac{x}{|x|}\right)^{2}+\frac{1}{|x-c|}\left(\xi\cdot\frac{x-c}{|x-c|}\right)^{2}\leq\lambda_{2}|(\hat{\xi}_{1},\hat{\xi}_{2})|^{2}\leq\lambda_{2}|\xi|^{2},
\]
thus we get, combining this last inequality with (\ref{eq:d2xi1})
and (\ref{eq:d2xi2}), for $x\in\Omega\cap B(0,A)$
\begin{equation}
\sin^{2}\theta\geq\alpha\implies(D^{2}\chi\cdot\xi,\xi)\gtrsim\alpha|\xi|^{2}.\label{eq:contrang}
\end{equation}
Remark that, because $\theta$ is the angle between $\frac{x}{|x|}$
and $\frac{x-c}{|x-c|}$
\[
\theta=\arccos\frac{x}{|x|}\cdot\frac{x-c}{|x-c|},
\]
and let us denote, for $\alpha\leq\alpha_{0}$ 
\[
V(\alpha)=\Omega\cap B(0,A)\cap\left\{ \sin^{2}\left(\arccos\frac{x}{|x|}\cdot\frac{x-c}{|x-c|}\right)\geq\alpha\right\} .
\]

\begin{figure}
\includegraphics[scale=0.7]{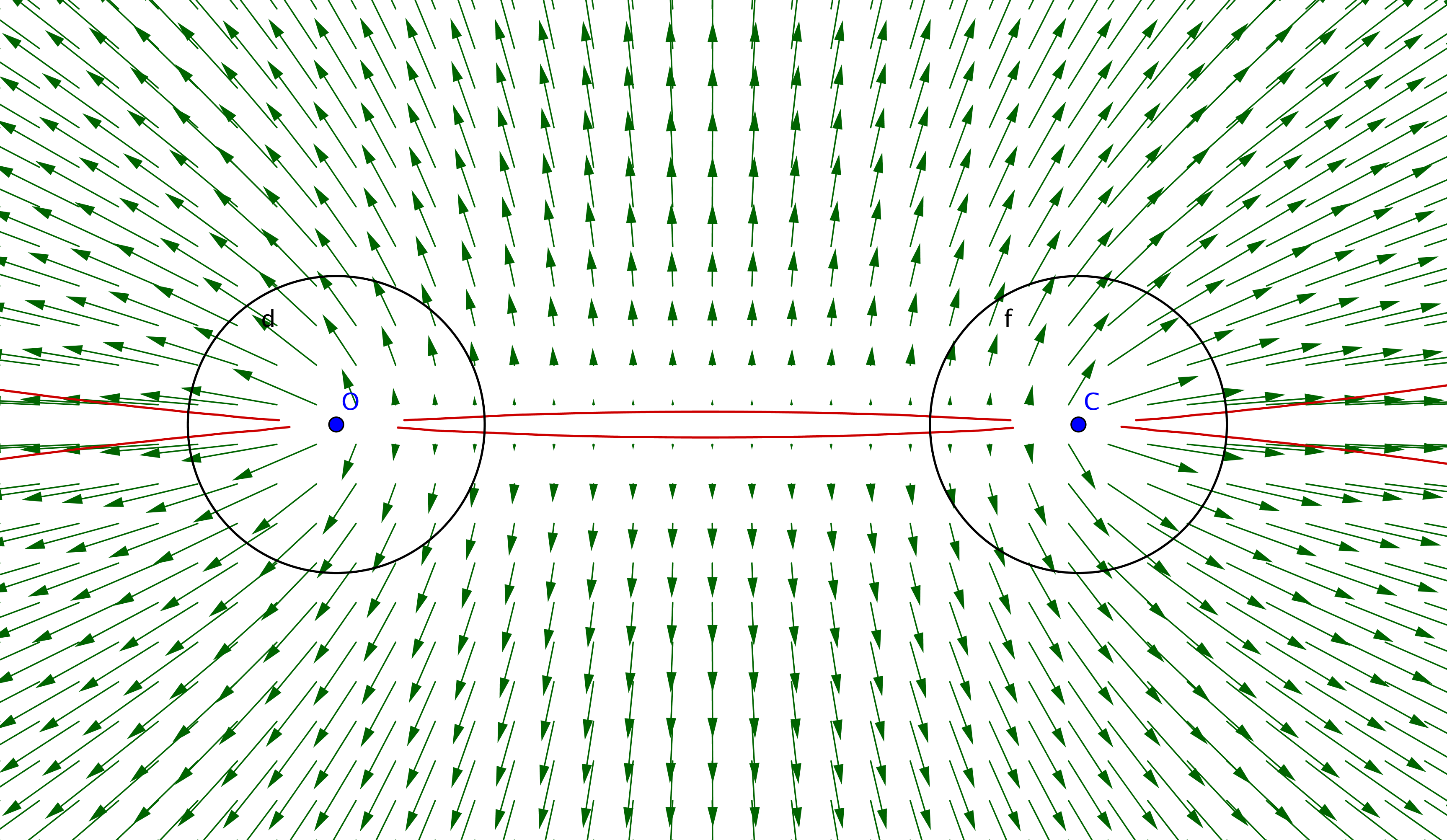}\caption{$\nabla\chi$ and $V(\alpha)$}
\end{figure}

Note that $V(\alpha)\rightarrow\Omega\cap B(0,A)$ as $\alpha$ goes
to zero in the sense that, denoting $\mu$ the Lebesgue's mesure on
$\mathbb{R}^{3}$
\begin{equation}
\mu\left((\Omega\cap B(0,A))\backslash V(\alpha)\right)\longrightarrow0\label{eq:set0}
\end{equation}
as $\alpha$ goes to zero: $V(\alpha)$ is a subset of $\Omega\cap B(0,A)$
excluding a small neighborhood of the line $(0,c)$. We have, on $S(\alpha)$,
because of (\ref{eq:contrang})
\[
(D^{2}\chi\cdot\xi,\xi)\gtrsim\alpha|\xi|^{2}.
\]
Thus we get
\[
\int_{\Omega}(D^{2}\chi\nabla u,\nabla u)\geq\int_{V(\alpha)}(D^{2}\chi\nabla u,\nabla u)\gtrsim\alpha\int_{V(\alpha)}|\nabla u|^{2}
\]
and by (\ref{eq:eq1}) we obtain
\[
\frac{1}{T}\int_{0}^{T}\int_{V(\alpha)}|\nabla u|^{2}\ dxdt\lesssim\frac{1}{\alpha}\frac{E}{T}.
\]
We take $\alpha=T^{-1/2}$ in order to have
\begin{equation}
\frac{1}{T}\int_{0}^{T}\int_{V(\alpha(T))}|\nabla u|^{2}\ dxdt\lesssim\frac{E}{\sqrt{T}}.\label{eq:decaygrad}
\end{equation}

Choosing $\mathcal{S}(T):=B(0,A)\backslash V(\alpha(T))$, (\ref{eq:decayl6})
together with (\ref{eq:decaygrad}) gives the result.
\end{proof}

\subsection{Obstacles illuminated by an ellipsoïd}

Motivated by the above result, we are interested by the scattering
problem in non trapping geometries close to the exterior of two convex
obstacles, such as dog bones with arbitrary thin neck. Theorem \ref{th3},
which we will prove now, gives in particular the scattering in such
settings. More precisely, it permits to handle obstacles illuminated
by arbitary cigar-shaped ellipsoïds and a certain class of flat-shaped
ones. 

In order to show such a result using the Morawetz identity (\ref{eq:mor})
to obtain the criterion of \lemref{contrenerg}, it is natural to
choose the gauge of the ellipsoïd we are dealing with as the weight
function. The next lemma gives us the range of ellipsoïds for which
such a weight verify the bilaplacian constraint:
\begin{lem}
\label{calcbile}Let $n\geq2$ and
\[
\rho(x)=\sqrt{x_{1}^{2}+\cdots+x_{k}^{2}+\epsilon(x_{k+1}^{2}+\cdots+x_{n}^{2})}.
\]
Then,
\[
\Delta^{2}\rho\leq0,\ \forall\epsilon\in[\epsilon_{0},1]
\]
with 
\[
\epsilon_{0}=\begin{cases}
0 & \text{if }k\geq3,\\
\frac{1}{n}+\frac{\sqrt{2(n-2)(n-1)}}{n(n-2)} & \text{if }k=2,\\
\frac{4}{n+1} & \text{if }k=1.
\end{cases}
\]
\end{lem}
\begin{proof}
An elemantary computation gives
\[
\Delta^{2}\rho(x)=\frac{A(\epsilon)}{\rho^{3}}+\frac{B(\epsilon)(x_{k+1}^{2}+\cdots+x_{n}^{2})}{\rho^{5}}+\frac{C(\epsilon)(x_{k+1}^{2}+\cdots+x_{n}^{2})^{2}}{\rho^{7}}
\]
where
\begin{gather*}
A(\epsilon)=-(n-k+2)(n-k)\epsilon^{2}-2(n-k)(k-3)\epsilon-(k-1)(k-3),\\
B(\epsilon)=6\epsilon((n-k+2)\epsilon^{2}+[(2k-n)-5]\epsilon-k+3),\\
C(\epsilon)=-15\epsilon^{2}(\epsilon-1)^{2}.
\end{gather*}
In all cases, $C(\epsilon)\le0$. 

If $k\geq3$, it is clear that $A(\epsilon)\leq0$. Let us denote
$\tilde{B}(\epsilon)=\frac{B(\epsilon)}{6\epsilon}=(n-k+2)\epsilon^{2}+[(2k-n)-5]\epsilon-k+3$
. Then $\tilde{B}(0)=-(k-3)\leq0$, $\tilde{B}(1)=0$, so $B(\epsilon)\leq0$
for $\epsilon\in[0,1]$ and therefore $\Delta^{2}\rho\leq0$ for $\epsilon\in[0,1]$. 

If $k=2$, the roots of $A(\epsilon)=-n(n-2)\epsilon^{2}+2(n-2)\epsilon+1$
are $\epsilon_{1,2}=\frac{1}{n}\pm\frac{\sqrt{2(n-2)(n-1)}}{n(n-2)}$
and the roots of $\tilde{B}(\epsilon)=n\epsilon^{2}-(n+1)\epsilon+1$
are $\frac{1}{n}$ and $1$, so $\Delta^{2}\rho\leq0$ for $\epsilon\in[\frac{1}{n}+\frac{\sqrt{2(n-2)(n-1)}}{n(n-2)},1]$.

If $k=1$, $A(\epsilon)=-(n+1)(n-1)\epsilon^{2}+4(n-1)\epsilon$ is
non-positive if and only if $\epsilon\geq\frac{4}{n+1}$, and the
roots of $\tilde{B}(\epsilon)=(n+1)\epsilon^{2}-(n+3)\epsilon+2$
are $\frac{2}{n+1}$ and $1$, so $\Delta^{2}\rho\leq0$ for all $\epsilon$
in $[\frac{4}{n+1},1]$.
\end{proof}
Notice that in dimension three, this weight cannot be explicitly used
as his bilaplacian is not non positive. Hence, to derive the control
(\ref{eq:contrenerg}), we will extend $u$ as a solution of a four
dimensional non linear wave equation, get the control for this extended
solution, and manage to go back to $u$. 

To this purpose, we need the following existence result for the four
dimensional problem:
\begin{lem}
\label{ext}Let $\Omega$ be a smooth domain of $\mathbb{R}^{4}$
with compact boundary, and $(u_{0},u_{1})\in\dot{H}^{\frac{7}{4}}(\Omega)\times H^{\frac{3}{4}}(\Omega)$.
Then, there exists a unique global solution of 

\begin{equation}
\begin{cases}
(\partial_{t}^{2}-\Delta)u+u^{5}=0\quad\text{in }\mathbb{R}\times\Omega\\
u_{|t=0}=u_{0},\quad\partial_{t}u_{|t=0}=u_{1},\quad u_{|\mathbb{R}\times\partial\Omega}=0
\end{cases}\label{eq:nlw-1-1}
\end{equation}
satisfying
\[
u\in C(\mathbb{R},\dot{H}^{\frac{7}{4}}(\Omega)\cap L^{6}(\Omega))\cap C^{1}(\mathbb{R},H^{\frac{3}{4}}(\Omega))\cap L^{48}(\mathbb{R},L^{6}(\Omega)).
\]
\end{lem}
\begin{proof}
Let $0<T<1$. By the work of \cite{MR2566711}, Theorem 1.1, applied
to the admissible triple $(p=48,q=6,\gamma=\frac{7}{4})$ in dimension
$4$, if $u$ is solution of 
\begin{equation}
\begin{cases}
(\partial_{t}^{2}-\Delta)u=F\quad\text{in }(0,T)\times\Omega\\
u_{|t=0}=f,\quad\partial_{t}u_{|t=0}=g,\quad u_{|\mathbb{R}\times\partial\Omega}=0
\end{cases}\label{eq:nlw-1-1-1}
\end{equation}
then the following Strichartz estimate holds
\[
\Vert u\Vert_{L^{48}((0,T),L^{6}(\Omega))}\leq C\left(\Vert f\Vert_{\dot{H}^{\frac{7}{4}}(\Omega)}+\Vert g\Vert_{H^{\frac{3}{4}}(\Omega)}+\Vert F\Vert_{L^{1}((0,T),L^{2}(\Omega))}\right).
\]

Using this estimate, we obtain the local existence following a classical
fixed point method, in the space
\[
X_{T}=C^{0}(\mathbb{R},\dot{H}^{\frac{7}{4}}(\Omega)\cap L^{6}(\Omega))\cap C^{1}(\mathbb{R},H^{\frac{3}{4}}(\Omega))\cap L^{48}(\mathbb{R},L^{6}(\Omega))
\]
for $T$ sufficently small depending of $\Vert u_{0}\Vert_{\dot{H}^{\frac{7}{4}}(\Omega)}$
and $\Vert u_{1}\Vert_{H^{\frac{3}{4}}(\Omega)}$. The global existence
result follows using the energy conservation law.
\end{proof}
We are now in position to show:
\begin{prop}
\label{ln}Let $\mathcal{C}\subset\mathbb{R}^{3}$ be the ellipsoïd
of equation (\ref{eq:el1}), resp. (\ref{eq:el2}), $\mathcal{K}$
be a compact subset of $\mathbb{R}^{3}$ illuminated by $\mathcal{C}$
and $\Omega:=\mathbb{R}^{3}\backslash\mathcal{K}$. Let $u\in C(\mathbb{R},\dot{H}^{1}(\Omega))\cap C^{1}(\mathbb{R},L^{2}(\Omega))$
be the global solution of (\ref{eq:nlw}) in $\Omega$. Then

\[
\frac{1}{T}\int_{0}^{T}\int_{\Omega\cap B(0,A)}|\nabla u(x,t)|^{2}+|u(x,t)|^{6}\ dxdt\leq\frac{1}{\ln T}C(E(u)).
\]
\end{prop}
\begin{proof}
Let $\delta>0$. There exists $u_{0}^{\delta}$, $u_{1}^{\delta}$,
smooth functions vanishing on $\partial\Omega$ such that
\[
\Vert u_{0}-u_{0}^{\delta}\Vert_{\dot{H}^{1}(\Omega)}+\Vert u_{1}-u_{1}^{\delta}\Vert_{L^{2}(\Omega)}\leq\delta.
\]
We denote by $u^{\delta}\in C^{0}(\mathbb{R},\dot{H}^{1}(\Omega))\cap C^{1}(\mathbb{R},L^{2}(\Omega))$
the solution of (\ref{eq:nlw}) in $\Omega$ with data $(u_{0}^{\delta},u_{1}^{\delta})$.

Let $T>0$ and $\phi\in C_{c}^{\infty}(\mathbb{R})$ be such that
$0\leq\phi\leq1$, $\phi=1$ on $[-1,1]$ and $\phi=0$ on $[-2,2]^{c}$.
We take $\chi_{T}=\phi(\frac{\cdot}{2T})$ and, for $(x,z)\in\Omega\times\mathbb{R}$
\begin{gather*}
v_{0}^{\delta}(x,z)=u_{0}^{\delta}(x)\chi_{T}(z)\chi_{T}(|x|),\\
v_{1}^{\delta}(x,z)=u_{1}^{\delta}(z)\chi_{T}(z)\chi_{T}(|x|).
\end{gather*}
Let us denote by $\rho$ the gauge of ellipsoïd we are dealing with,
consider
\[
\tilde{\mathcal{K}}=\mathcal{K}\times[-4T,4T],\ \tilde{\Omega}=\mathbb{R}^{4}\backslash\tilde{\mathcal{K}}
\]
\[
\tilde{\rho}(x,z)=\sqrt{\rho(x)^{2}+z^{2}},
\]
and $v^{\delta}\in C(\mathbb{R},H^{\frac{7}{4}}(\tilde{\Omega}))\cap C^{1}(\mathbb{R},H^{\frac{3}{4}}(\tilde{\Omega}))$
the solution of the four dimensional equation
\begin{equation}
\begin{cases}
(\partial_{t}^{2}-\Delta)v^{\delta}+(v^{\delta})^{5}=0\quad\text{in }\mathbb{R}\times\Omega\\
v_{|t=0}^{\delta}=v_{0}^{\delta},\quad\partial_{t}v_{|t=0}^{\delta}=v_{1}^{\delta},\quad v_{|\mathbb{R}\times\partial\Omega}^{\delta}=0
\end{cases}\label{eq:nlw-1-1-2}
\end{equation}
given by lemma \ref{ext}.

Notice that, by the finite speed of propagation, 
\begin{gather}
v^{\delta}(x,z,t)=u^{\delta}(x,t),\label{eq:fsp}\\
\forall x\in\Omega\cap B(0,2T-t),\;\forall t\in[0,2T[,\;\forall z\in[-2T+t,2T-t].\nonumber 
\end{gather}
We denote by $n$ the outward pointing normal vector to $\partial\mathcal{K}$
and $\tilde{n}=(n,0)$ the outward pointing normal vector to $\partial\tilde{\mathcal{K}}$.
The momentum derivation (\ref{eq:mor}) applied to $v^{\delta}$ gives
\begin{multline}
\frac{d}{dt}\left(-\int_{\tilde{\Omega}}\partial_{t}v^{\delta}\nabla v^{\delta}\cdot\nabla\tilde{\rho}-\frac{1}{2}\int_{\tilde{\Omega}}\Delta\tilde{\rho}v^{\delta}\partial_{t}v^{\delta}\right)=\int_{\tilde{\Omega}}(D^{2}\tilde{\rho}\nabla v^{\delta},\nabla v^{\delta})+\frac{1}{4}\Delta\tilde{\rho}|v^{\delta}|^{6}-\Delta^{2}\tilde{\rho}|v^{\delta}|^{2}\\
+\frac{1}{2}\int_{\partial\tilde{V}}|\partial_{\tilde{n}}v^{\delta}|^{2}\partial_{\tilde{n}}\tilde{\rho}.\label{eq:41}
\end{multline}
By lemma \ref{calcbile}, $\Delta^{2}\tilde{\rho}\leq0$. Moreover,
as $\tilde{\rho}$ is convex, $(D^{2}\tilde{\rho}\nabla v^{\delta},\nabla v^{\delta})\geq0$
and $\Delta\tilde{\rho}\geq0$. Therefore, integrating (\ref{eq:41})
we obtain
\begin{equation}
\frac{1}{2}\int_{0}^{T}\int_{\partial\tilde{K}}|\partial_{\tilde{n}}v^{\delta}|^{2}\partial_{\tilde{n}}\tilde{\rho}\:d\sigma dt\leq\left[-\int_{\tilde{\Omega}}\partial_{t}v^{\delta}\nabla v^{\delta}\cdot\nabla\tilde{\rho}-\frac{1}{2}\int_{\tilde{\Omega}}\Delta\tilde{\rho}v^{\delta}\partial_{t}v^{\delta}\right]_{0}^{T}.\label{eq:42}
\end{equation}
Let us take $T>0$ large enought so that $\partial V\subset\Omega\cap B(0,T)$.
Then, for $t\in[0,T]$,
\[
\int_{\partial\tilde{K}}|\partial_{\tilde{n}}v^{\delta}|^{2}\partial_{\tilde{n}}\tilde{\rho}\,d\tilde{\sigma}=\int_{\mathbb{R}}\int_{\partial K}|\partial_{\tilde{n}}v^{\delta}|^{2}\partial_{\tilde{n}}\tilde{\rho}\,d\sigma dz\geq\int_{-T}^{T}\int_{\partial K}|\partial_{\tilde{n}}v^{\delta}|^{2}\partial_{\tilde{n}}\tilde{\rho}\,d\sigma dz.
\]
But, by finite speed of propagation (\ref{eq:fsp}),
\[
\int_{-T}^{T}\int_{\partial K}|\partial_{\tilde{n}}v^{\delta}|^{2}\partial_{\tilde{n}}\tilde{\rho}\,d\sigma dz=\int_{-T}^{T}\int_{\partial K}|\partial_{n}u^{\delta}|^{2}\partial_{\tilde{n}}\tilde{\rho}\,d\sigma dz,
\]
so we obtain, as $\partial_{n}\rho\geq C$ by the definition of an
illuminated subset 
\begin{multline*}
\int_{\partial\tilde{K}}|\partial_{\tilde{n}}v^{\delta}|^{2}\partial_{\tilde{n}}\tilde{\rho}\,d\tilde{\sigma}\geq\int_{-T}^{T}\int_{\partial K}|\partial_{\tilde{n}}u^{\delta}|^{2}\partial_{\tilde{n}}\tilde{\rho}\,d\sigma dz\\
=\int_{-T}^{T}\int_{\partial K}|\partial_{n}u^{\delta}|^{2}\frac{\rho}{\sqrt{\rho^{2}+z^{2}}}\partial_{n}\mathbb{\rho}\,d\sigma dz\gtrsim\int_{\partial K}|\partial_{n}u^{\delta}|^{2}\int_{-T}^{T}\frac{1}{\sqrt{1+z^{2}}}\,dzd\sigma,
\end{multline*}
and the integration of the right hand side gives
\begin{equation}
\int_{\partial K}|\partial_{n}u^{\delta}|^{2}\:d\sigma\lesssim\frac{1}{\ln T}\int_{\partial\tilde{K}}|\partial_{\tilde{n}}v^{\delta}|^{2}\partial_{\tilde{n}}\tilde{\rho}\,d\tilde{\sigma}.\label{eq:45}
\end{equation}
Moreover, as $\nabla\tilde{\rho}$ is bounded and $0\leq\Delta\tilde{\rho}\leq\frac{C}{|(x,z)|}$,
by the Cauchy-Schwarz inequality and the Hardy inequality 
\[
\int_{\Omega}\frac{|f|^{2}}{|x|^{2}}\lesssim\int_{\Omega}|\nabla f|^{2}\ \text{for }f\in\dot{H}_{0}^{1}(\Omega)
\]
we obtain
\begin{equation}
\big|\left[-\int_{\tilde{\Omega}}\partial_{t}v^{\delta}\nabla v^{\delta}\cdot\nabla\tilde{\rho}-\frac{1}{2}\int_{\tilde{\Omega}}\Delta\tilde{\rho}v^{\delta}\partial_{t}v^{\delta}\right]_{0}^{T}\big|\lesssim E(v^{\delta}),\label{eq:46}
\end{equation}
and this last inequality combined to (\ref{eq:42}) and (\ref{eq:45})
gives
\begin{equation}
\int_{0}^{T}\int_{\partial\Omega}(\partial_{n}u^{\delta})^{2}\:d\sigma dt\lesssim\frac{1}{\ln T}E(v^{\delta}).\label{eq:47}
\end{equation}

It remains to estimate the energy of $v^{\delta}$. We have
\begin{eqnarray*}
\int_{\tilde{\Omega}}|\nabla v_{0}^{\delta}|^{2} & = & \int_{\tilde{\Omega}}|\nabla u_{0}^{\delta}(x)\chi_{T}(z)\chi_{T}(|x|)+u_{0}^{\delta}(x)\chi_{T}(z)\chi_{T}'(|x|)\cdot\frac{x}{|x|}|^{2}\,dxdz\\
 &  & +\int_{\tilde{\Omega}}|u_{0}^{\delta}(x)|^{2}\chi_{T}'(z)^{2}\chi_{T}(|x|)^{2}\,dxdz\\
 & \lesssim & \int_{\tilde{\Omega}}|\nabla u_{0}^{\delta}(x)|^{2}\chi_{T}(z)^{2}\chi_{T}(|x|)^{2}\,dxdz+\int_{\tilde{\Omega}}|u_{0}^{\delta}(x)|^{2}\chi(z)^{2}\chi'(|x|){}^{2}\,dxdz\\
 &  & \qquad+\int_{\tilde{\Omega}}|u_{0}^{\delta}(x)|^{2}\chi_{T}'(z)^{2}\chi_{T}(|x|)^{2}\,dxdz,
\end{eqnarray*}
so, by the Hölder inequality 
\begin{eqnarray*}
\int_{\tilde{\Omega}}|\nabla v_{0}^{\delta}|^{2} & \lesssim & \int_{\Omega}|\nabla u_{0}^{\delta}(x)|^{2}dx\int_{\mathbb{R}}\chi_{T}(z)^{2}\,dz+\left(\int_{\Omega}|u_{0}^{\delta}(x)|^{6}dx\right)^{\frac{1}{3}}\left(\int_{\Omega}\chi_{T}'(|x|)^{3}dx\right)^{\frac{2}{3}}\int_{\mathbb{R}}\chi_{T}(z)^{2}dz\\
 &  & \qquad+\left(\int_{\Omega}|u_{0}^{\delta}(x)|^{6}dx\right)^{\frac{1}{3}}\left(\int_{\Omega}\chi_{T}(|x|)^{3}dx\right)^{\frac{2}{3}}\int_{\mathbb{R}}\chi_{T}'(z)^{2}dz\\
 & \lesssim & \int_{\Omega}|\nabla u_{0}^{\delta}|^{2}\int_{-4T}^{4T}\Vert\phi\Vert_{\infty}^{2}+\left(\int_{\Omega}|u_{0}^{\delta}|^{6}\right)^{\frac{1}{3}}\left(\int_{B(0,4T)}\left(\frac{1}{2T}\Vert\phi'\Vert_{L^{\infty}}\right)^{3}\right)^{\frac{2}{3}}\int_{-4T}^{4T}\Vert\phi\Vert_{L^{\infty}}^{2}\\
 &  & \qquad+\left(\int_{\Omega}|u_{0}^{\delta}|^{6}\right)^{\frac{1}{3}}\left(\int_{B(0,4T)}\Vert\phi\Vert_{L^{\infty}}^{3}\right)^{\frac{2}{3}}\int_{-4T}^{4T}\frac{1}{4T^{2}}\Vert\phi'\Vert_{L^{\infty}}^{2}\\
 & \lesssim & T\int_{\Omega}|\nabla u_{0}^{\delta}|^{2}+T\left(\int_{\Omega}|u_{0}^{\delta}|^{6}\right)^{\frac{1}{3}}+T\left(\int_{\Omega}|u_{0}^{\delta}|^{6}\right)^{\frac{1}{3}}\\
 & \lesssim & TC(E(u^{\delta})).
\end{eqnarray*}
Moreover,
\[
\int_{\tilde{\Omega}}|v_{0}^{\delta}|^{6}=\int_{\Omega}\int_{\mathbb{R}}|u_{0}^{\delta}(x)|^{6}\chi_{T}(z)^{6}\chi_{T}(|x|)^{6}dxdz\leq4T\int_{\Omega}|u_{0}^{\delta}|^{6},
\]
and
\[
\int_{\tilde{\Omega}}|v_{1}^{\delta}|^{2}=\int_{\Omega}\int_{\mathbb{R}}|u_{1}^{\delta}(x)|^{2}\chi_{T}(z)^{2}\chi_{T}(|x|)^{2}dxdz\leq4T\int_{\Omega}|u_{1}^{\delta}|^{2},
\]
therefore 
\[
E(v^{\delta})\lesssim TC(E(u^{\delta})),
\]
and by (\ref{eq:46}) we obtain 
\begin{equation}
\int_{0}^{T}\int_{\partial K}|\partial_{n}u^{\delta}|^{2}\:d\sigma dt\lesssim\frac{T}{\ln T}C(E(u^{\delta})).\label{eq:prefinal}
\end{equation}

Notice that we can not pass to the limit directly in this expression
because as mentionned before, the application
\[
u\in\dot{H}^{1}\cap\left\{ \text{value in time \ensuremath{t} of solutions of NLW}\right\} \longrightarrow\partial_{n}u\in L^{2}(\partial\Omega)
\]
 is not known to be continuous. But, notice that using \lemref{bordBS}
with the weight $\chi=|x|^{2}$ gives in particular 
\[
\int_{0}^{T}\int_{\Omega\cap B(0,A)}|\nabla u^{\delta}(x,t)|^{2}+|u^{\delta}(x,t)|^{6}\ dxdt\lesssim\int_{0}^{T}\int_{\partial K}|\partial_{n}u^{\delta}|^{2}\:d\sigma dt.
\]
Therefore, combining the last inequality with (\ref{eq:prefinal})
we get
\[
\frac{1}{T}\int_{0}^{T}\int_{\Omega\cap B(0,A)}|\nabla u^{\delta}(x,t)|^{2}+|u^{\delta}(x,t)|^{6}\ dxdt\lesssim\frac{1}{\ln T}C(E(u^{\delta})),
\]
and we can let $\delta$ go to zero in this expression: as $u_{0}^{\delta}\underset{\delta\rightarrow0}{\longrightarrow}u_{0}$
in $\dot{H}^{1}$ and $u_{1}^{\delta}\underset{\delta\rightarrow0}{\longrightarrow}u_{1}$
in $L^{2}$, we obtain using the energy conservation law that $C(E(u^{\delta}))\underset{\delta\rightarrow0}{\longrightarrow}C(E(u)),$
and, because the problem (\ref{eq:nlw}) is well posed, the left hand
side goes as $\delta$ goes to zero to
\[
\frac{1}{T}\int_{0}^{T}\int_{\Omega\cap B(0,A)}|\nabla u(x,t)|^{2}+|u(x,t)|^{6}\ dxdt,
\]
and the proposition holds.
\end{proof}
Finally, we can conclude:
\begin{proof}[Proof of Theorem \ref{th3}.]
 The above proposition combined with the scattering criterion of
\lemref{contrenerg} gives immediately the result.
\end{proof}

\subsection*{Aknowledgment. }

The author would like to thanks Nicolas Burq for helpful discussions,
and Fabrice Planchon for enlightening discussions about multipliers
methods.

\bibliographystyle{amsalpha}
\bibliography{megaref}

\providecommand{\bysame}{\leavevmode\hbox to3em{\hrulefill}\thinspace}
\providecommand{\MR}{\relax\ifhmode\unskip\space\fi MR }
\providecommand{\MRhref}[2]{%
  \href{http://www.ams.org/mathscinet-getitem?mr=#1}{#2}
}
\providecommand{\href}[2]{#2}
\begin{thebibliography}{BSS09b}

\bibitem[ANV04]{boundedlp}
Ryuichi Ashino, Michihiro Nagase, and R{\'e}mi Vaillancourt,
  \emph{Pseudodifferential operators in {$L^p(\Bbb R^n)$} spaces}, Cubo
  \textbf{6} (2004), no.~3, 91--129. \MR{2124828}

\bibitem[AS13]{farah}
Farah Abou~Shakra, \emph{Asymptotics of the critical nonlinear wave equation
  for a class of non-star-shaped obstacles}, J. Hyperbolic Differ. Equ.
  \textbf{10} (2013), no.~3, 495--522. \MR{3104079}

\bibitem[BGH10]{MR2720226}
Nicolas Burq, Colin Guillarmou, and Andrew Hassell, \emph{Strichartz estimates
  without loss on manifolds with hyperbolic trapped geodesics}, Geom. Funct.
  Anal. \textbf{20} (2010), no.~3, 627--656. \MR{2720226 (2012f:58068)}

\bibitem[BK74]{MR0347237}
Clifford~O. Bloom and Nicholas~D. Kazarinoff, \emph{Local energy decay for a
  class of nonstar-shaped bodies}, Arch. Rational Mech. Anal. \textbf{55}
  (1974), 73--85. \MR{0347237 (49 \#11957)}

\bibitem[BLP08]{BLP}
Nicolas Burq, Gilles Lebeau, and Fabrice Planchon, \emph{Global existence for
  energy critical waves in 3-{D} domains}, J. Amer. Math. Soc. \textbf{21}
  (2008), no.~3, 831--845. \MR{2393429}

\bibitem[BS98]{BahSha}
Hajer Bahouri and Jalal Shatah, \emph{Decay estimates for the critical
  semilinear wave equation}, Ann. Inst. H. Poincar\'e Anal. Non Lin\'eaire
  \textbf{15} (1998), no.~6, 783--789. \MR{1650958}

\bibitem[BSS09a]{BSS}
Matthew~D. Blair, Hart~F. Smith, and Christopher~D. Sogge, \emph{Strichartz
  estimates for the wave equation on manifolds with boundary}, Ann. Inst. H.
  Poincar\'e Anal. Non Lin\'eaire \textbf{26} (2009), no.~5, 1817--1829.
  \MR{2566711}

\bibitem[BSS09b]{MR2566711}
\bysame, \emph{Strichartz estimates for the wave equation on manifolds with
  boundary}, Ann. Inst. H. Poincar\'e Anal. Non Lin\'eaire \textbf{26} (2009),
  no.~5, 1817--1829. \MR{2566711 (2010k:35500)}

\bibitem[Bur93]{plaques}
Nicolas Burq, \emph{Contr\^ole de l'\'equation des plaques en pr\'esence
  d'obstacles strictement convexes}, M\'em. Soc. Math. France (N.S.) (1993),
  no.~55, 126. \MR{1254820}

\bibitem[Bur98]{MR1618254}
\bysame, \emph{D\'ecroissance de l'\'energie locale de l'\'equation des ondes
  pour le probl\`eme ext\'erieur et absence de r\'esonance au voisinage du
  r\'eel}, Acta Math. \textbf{180} (1998), no.~1, 1--29. \MR{1618254
  (99j:35119)}

\bibitem[Bur03]{MR2001179}
N.~Burq, \emph{Global {S}trichartz estimates for nontrapping geometries: about
  an article by {H}. {F}.\ {S}mith and {C}. {D}.\ {S}ogge: ``{G}lobal
  {S}trichartz estimates for nontrapping perturbations of the {L}aplacian''
  [{C}omm. {P}artial {D}ifferential {E}quation {\bf 25} (2000), no. 11-12
  2171--2183; {MR}1789924 (2001j:35180)]}, Comm. Partial Differential Equations
  \textbf{28} (2003), no.~9-10, 1675--1683. \MR{2001179}

\bibitem[Bur04]{MR2066943}
\bysame, \emph{Smoothing effect for {S}chr\"odinger boundary value problems},
  Duke Math. J. \textbf{123} (2004), no.~2, 403--427. \MR{2066943
  (2006e:35026)}

\bibitem[GV85a]{GV85}
J.~Ginibre and G.~Velo, \emph{Scattering theory in the energy space for a class
  of nonlinear {S}chr\"odinger equations}, J. Math. Pures Appl. (9) \textbf{64}
  (1985), no.~4, 363--401. \MR{839728}

\bibitem[GV85b]{GinibreVeloKG85}
\bysame, \emph{Time decay of finite energy solutions of the nonlinear
  {K}lein-{G}ordon and {S}chr\"odinger equations}, Ann. Inst. H. Poincar\'e
  Phys. Th\'eor. \textbf{43} (1985), no.~4, 399--442. \MR{824083}

\bibitem[Ika82]{IkawaMult}
Mitsuru Ikawa, \emph{Decay of solutions of the wave equation in the exterior of
  two convex obstacles}, Osaka J. Math. \textbf{19} (1982), no.~3, 459--509.
  \MR{676233}

\bibitem[Ika88]{Ikawa2}
\bysame, \emph{Decay of solutions of the wave equation in the exterior of
  several convex bodies}, Ann. Inst. Fourier (Grenoble) \textbf{38} (1988),
  no.~2, 113--146. \MR{949013}

\bibitem[ILLP]{ILLPGeneral}
Oana Ivanovici, Richard Lascar, Gilles Lebeau, and Fabrice Planchon,
  \emph{Dispersion for the wave equation inside strictly convex domains {II}:
  the general case}, Preprint.

\bibitem[ILP14]{ILPAnnals}
Oana Ivanovici, Gilles Lebeau, and Fabrice Planchon, \emph{Dispersion for the
  wave equation inside strictly convex domains {I}: the {F}riedlander model
  case}, Ann. of Math. (2) \textbf{180} (2014), no.~1, 323--380. \MR{3194817}

\bibitem[Iva12]{OanaCounterex}
Oana Ivanovici, \emph{Counterexamples to the {S}trichartz inequalities for the
  wave equation in general domains with boundary}, J. Eur. Math. Soc. (JEMS)
  \textbf{14} (2012), no.~5, 1357--1388. \MR{2966654}

\bibitem[Kap89]{Kapitanskii}
L.~V. Kapitanski{\u\i}, \emph{Some generalizations of the
  {S}trichartz-{B}renner inequality}, Algebra i Analiz \textbf{1} (1989),
  no.~3, 127--159. \MR{1015129}

\bibitem[KT98]{KeelTao}
Markus Keel and Terence Tao, \emph{Endpoint {S}trichartz estimates}, Amer. J.
  Math. \textbf{120} (1998), no.~5, 955--980. \MR{1646048}

\bibitem[Laf17]{Schreodinger}
D.~Lafontaine, \emph{Strichartz estimates without loss outside two strictly
  convex obstacles}, Preprint, https://arxiv.org/abs/1709.03836 (2017).

\bibitem[LL]{LafLaurent}
D.~Lafontaine and C.~Laurent, \emph{Scattering for critical nonlinear waves
  outside some strictly convex obstacles}, Work in progress.

\bibitem[LS95]{LindbladSogge}
Hans Lindblad and Christopher~D. Sogge, \emph{On existence and scattering with
  minimal regularity for semilinear wave equations}, J. Funct. Anal.
  \textbf{130} (1995), no.~2, 357--426. \MR{1335386}

\bibitem[Met04]{Metcalfe}
Jason~L. Metcalfe, \emph{Global {S}trichartz estimates for solutions to the
  wave equation exterior to a convex obstacle}, Trans. Amer. Math. Soc.
  \textbf{356} (2004), no.~12, 4839--4855. \MR{2084401}

\bibitem[Mor61]{Morawetz61}
Cathleen~S. Morawetz, \emph{The decay of solutions of the exterior
  initial-boundary value problem for the wave equation}, Comm. Pure Appl. Math.
  \textbf{14} (1961), 561--568. \MR{0132908}

\bibitem[MS78]{MR0492794}
R.~B. Melrose and J.~Sj{\"o}strand, \emph{Singularities of boundary value
  problems. {I}}, Comm. Pure Appl. Math. \textbf{31} (1978), no.~5, 593--617.
  \MR{0492794 (58 \#11859)}

\bibitem[MS82]{MR644020}
\bysame, \emph{Singularities of boundary value problems. {II}}, Comm. Pure
  Appl. Math. \textbf{35} (1982), no.~2, 129--168. \MR{644020 (83h:35120)}

\bibitem[MSS93]{MoSeSo}
Gerd Mockenhaupt, Andreas Seeger, and Christopher~D. Sogge, \emph{Local
  smoothing of {F}ourier integral operators and {C}arleson-{S}j\"olin
  estimates}, J. Amer. Math. Soc. \textbf{6} (1993), no.~1, 65--130.
  \MR{1168960}

\bibitem[Smi98]{SmithC11}
Hart~F. Smith, \emph{A parametrix construction for wave equations with
  {$C^{1,1}$} coefficients}, Ann. Inst. Fourier (Grenoble) \textbf{48} (1998),
  no.~3, 797--835. \MR{1644105}

\bibitem[SS95]{SS95}
Hart~F. Smith and Christopher~D. Sogge, \emph{On the critical semilinear wave
  equation outside convex obstacles}, J. Amer. Math. Soc. \textbf{8} (1995),
  no.~4, 879--916. \MR{1308407}

\bibitem[SS00]{SmithSoggeNonTrapping}
\bysame, \emph{Global {S}trichartz estimates for nontrapping perturbations of
  the {L}aplacian}, Comm. Partial Differential Equations \textbf{25} (2000),
  no.~11-12, 2171--2183. \MR{1789924}

\bibitem[SS07]{SSspectral}
\bysame, \emph{On the {$L^p$} norm of spectral clusters for compact manifolds
  with boundary}, Acta Math. \textbf{198} (2007), no.~1, 107--153. \MR{2316270}

\bibitem[Str77]{Strichartz}
Robert~S. Strichartz, \emph{Restrictions of {F}ourier transforms to quadratic
  surfaces and decay of solutions of wave equations}, Duke Math. J. \textbf{44}
  (1977), no.~3, 705--714. \MR{0512086}

\bibitem[Tat02]{Tataruns}
Daniel Tataru, \emph{Strichartz estimates for second order hyperbolic operators
  with nonsmooth coefficients. {III}}, J. Amer. Math. Soc. \textbf{15} (2002),
  no.~2, 419--442. \MR{1887639}

\bibitem[VZ00]{MR1764368}
Andr{\'a}s Vasy and Maciej Zworski, \emph{Semiclassical estimates in
  asymptotically {E}uclidean scattering}, Comm. Math. Phys. \textbf{212}
  (2000), no.~1, 205--217. \MR{1764368 (2002b:58047)}

\end{thebibliography}

\end{document}